\documentclass[11pt]{article}
\usepackage{amsmath,amsthm,amssymb,hyperref}
\usepackage{color} % For TODO's
\usepackage[margin=1.3in]{geometry}

\hypersetup{
pdfauthor = {Boris Bukh and Jacob Tsimerman},
pdftitle = {Sum-product estimates for rational functions},
pdfsubject = {Mathematics},
pdfkeywords = {sum-product estimates, fiber product, Schwartz--Zippel lemma},
pdfstartview = {FitH},
pdfpagemode = {None}}
\author{Boris Bukh\footnote{\texttt{B.Bukh@dpmms.cam.ac.uk}. 
Centre for Mathematical Sciences,
Cambridge CB3 0WB, England
and Churchill College, Cambridge CB3 0DS, England.}
\and Jacob Tsimerman\footnote{\texttt{jtsimerm@math.princeton.edu}. 
Department of Mathematics, Princeton University,
Princeton, NJ 08544, USA}}
\title{Sum-product estimates for rational functions\footnote{\texttt{MSC classification:} 05E15, 11T23, 11B75, 14N10}}
\date{}
\newcommand*{\A}{\mathbb{A}}
\newcommand*{\F}{\mathbb{F}}
\newcommand*{\R}{\mathbb{R}}
\newcommand*{\K}{\mathbb{K}}
\newcommand*{\M}{\mathbb{M}}
\newcommand*{\J}{\mathbb{P}}
\newcommand*{\Lbb}{\mathbb{L}}

\theoremstyle{plain}
\newtheorem{theorem}{Theorem}
\newtheorem{lemma}[theorem]{Lemma}
\newtheorem{corollary}[theorem]{Corollary}

\newtheorem*{example}{Example}
\newtheorem*{conjecture}{Conjecture}

\theoremstyle{definition}
\newtheorem*{definition}{Definition}

\newcommand*{\abs}[1]{\lvert #1\rvert}                         % absolute value and cardinality
\newcommand*{\norm}[1]{\lVert #1\rVert}                        % norm
\newcommand*{\veps}{\varepsilon}                               % fancy epsilons
\newcommand*{\innq}[2]{\langle #2\rangle_{\F_q^{#1}}}          % inner product in F_q
\newcommand*{\inndq}[2]{\langle #2\rangle_{\hat{\F}_q^{#1}}}    % inner product in the dual of F_q
\DeclareMathOperator{\lspan}{span}                             % linear span

\begin{document}
\maketitle
\begin{abstract}
We establish several sum-product estimates over finite fields that 
involve polynomials and rational functions.

First, $\abs{f(A)+f(A)}+\abs{AA}$ is substantially larger 
than $\abs{A}$ for an arbitrary polynomial $f$ over $\F_p$. 
Second, a characterization is given for the rational functions $f$ 
and $g$ for  which $\abs{f(A)+f(A)}+\abs{g(A,A)}$ can be as small as $\abs{A}$, 
for large $\abs{A}$. 
Third, we show that under mild conditions on $f$, 
$\abs{f(A,A)}$ is substantially larger than 
$\abs{A}$, provided $\abs{A}$ is large.

We also present a conjecture on what the general sum-product result should be.
\end{abstract}
\section{Introduction and statement of the results}
\paragraph{Sum-product estimates}
For a polynomial $f(x_1,\dotsc,x_k)$ and sets $A_1,\dotsc,A_k$ define
$f(A_1,\dotsc,A_k)=\{f(a_1,\dotsc,a_k) : a_i\in A_i\}$.
As mostly we will be dealing with the case $A_1=\dotsb=A_k$, 
we write $f(A)=f(A,\dotsc,A)$ for brevity.
We also employ the notation $A+B=\{a+b : a\in A,\ b\in B\}$ and
$AB=\{ab : a\in A,\ b\in B\}$. The notation $X\ll Y$ means
$X\leq C Y$ for some effective absolute constant~$C$, whereas 
$X\ll_{r,s,\dotsc} Y$ means $X\leq C(r,s,\dotsc)Y$ for some function
$C$. 

The sum-product theorem of Erd\H{o}s and 
Szemer\'edi \cite{erdos_szemeredi} states 
that if $A$ is a finite set of real numbers, then either $A+A$ or
$AA$ has at least $\abs{A}^{1+c}$ elements where $c>0$ 
is an absolute constant. After several improvements, the current
record, due to Solymosi \cite{solymosi_real}, is that the result holds with 
$c=1/3-o(1)$, and it is conjectured that
$c=1-o(1)$ is admissible.

In practice, most applicable are the sum-product estimates when $A$
is in a finite field or a finite ring. Not only they have been
used to tackle a wide range of problems (see \cite{bourgain_survey}
for a survey), but they are also more general, for 
as was shown in \cite{vww_compactness} the uniform sum-product 
estimates in $\F_p$ imply the sum-product estimates over the complex 
numbers. The first estimate in $\F_p$ was proved by Bourgain, Katz 
and Tao \cite{bkt} for $\abs{A}\geq p^{\delta}$ for arbitrarily small, 
but fixed $\delta>0$. The restriction was subsequently removed 
by Bourgain and Konyagin \cite{bourgain_konyagin}. There was a rapid series
of improvements, with the best known bounds for 
$A\subset \F_p$ being
\begin{equation}\label{sumproductreview}
\abs{A+A}+\abs{AA}\gg 
\begin{cases}
\abs{A}^{12/11-o(1)},&
    \text{if }\abs{A}\leq p^{1/2},
    \text{(see \cite{rudnev_twelve})},\\
%% OLD BOUNDS
%\abs{A}^{13/12},&
%   \text{if }\abs{A}\leq p^{1/2}
%   \text{ (see \cite{li_slight} and \cite{bourgain_garaev})},\\
\abs{A}^{13/12}(\abs{A}/\sqrt{p})^{1/12-o(1)},&
   \text{if }p^{1/2}\leq \abs{A}\leq p^{35/68}
   \text{ (see \cite{li_slight})},\\
\abs{A}(p/\abs{A})^{1/11-o(1)},&
   \text{if }p^{35/68}\leq \abs{A}\leq p^{13/24}
   \text{ (see \cite{li_slight})},\\
\abs{A}\cdot \abs{A}/\sqrt{p},&
   \text{if }p^{13/24}\leq \abs{A}\leq p^{2/3}
   \text{ (see \cite{garaev_sharp})},\\
\abs{A}(p/\abs{A})^{1/2},&
   \text{if }\abs{A}>p^{2/3}
   \text{ (see \cite{garaev_sharp})}.
\end{cases}
\end{equation}
Of these results, Garaev's estimate of $\abs{A}(p/\abs{A})^{1/2}$
for $\abs{A}>p^{2/3}$ is 
notable in that it is the only sharp bound. It is likely
that $\abs{A+A}+\abs{AA}\gg 
\min(\abs{A}(p/\abs{A})^{1/2-o(1)},\abs{A}^{2-o(1)})$.

Of use are also the statements that one of $A+A$ or $f(A)$ is substantially
larger than $A$, where $f$ is a rational function, that is possibly different 
from $f(x,y)=xy$. For example, Bourgain
\cite{bourgain_potpourri} showed that either $A+A$ or $1/A+1/A$ is
always large, and used this estimate to give new bounds on certain
bilinear Kloosterman sums. In application to a construction of extractors,
in the same paper Bourgain asked for sum-product estimates
for $f(x_1,\dotsc,x_k)=x_1^t+\dotsb+x_k^t$. The most general result 
of the kind is due to Vu\cite{vu_expanders}, who generalized an earlier 
argument of Hart, Iosevich and Solymosi\cite{kloostermaniacs}. 
Call a polynomial 
$f(x,y)$ degenerate, if it is a function of a linear form in $x$ 
and~$y$. In \cite{vu_expanders} it was shown that if $f$ is a bivariate 
non-degenerate polynomial of degree $d$, then
\begin{equation}\label{vu_bound}
\abs{A+A}+\abs{f(A)}\gg 
\begin{cases} 
\abs{A}(\abs{A}/d^2 \sqrt{p})^{1/2}d^{-1},&
   \text{if }p^{1/2}<\abs{A}\leq d^{4/5}p^{7/10},\\
\abs{A}(p/\abs{A})^{1/3}d^{-1/3},&\text{if }\abs{A}\geq d^{4/5}p^{7/10}.
\end{cases}
\end{equation}
The same argument was used in \cite{hart_li_shen_fourieragain} to 
establish a version of \eqref{vu_bound} for $\abs{A+B}+\abs{f(A)}$.

As nearly all applications of sum-product estimates in finite fields
have taken advantage of validity of the estimates 
for $A$ of very small size, it is of interest to extend
this result to $\abs{A}<\sqrt{p}$. The result can also be improved
qualitatively because for some polynomials $f$ it is true that
$f(A)$ is much larger than $A$ no matter how large or small $A+A$ is.
So, for example, Bourgain \cite{bourgain_potpourri} showed that $x^2+xy$
and $x(y+a)$ for $a\neq 0$ are such polynomials
(Bourgain actually showed that $x^2+xy$ grows even if $x$ and $y$ range
over different sets; see also \cite{hegyvari_hennecart} for a generalization).

The sum-product estimate results are connected to the problem of giving good 
upper bounds on the number points in a Cartesian product set, such as 
$A \times A \times A$, that lie on a given variety. For example, the estimate 
\eqref{vu_bound} is related to a bound on the the number of points on a surface, 
in the case $A$ has small additive doubling (see 
Lemma \ref{lem_vararb} for the explicit form). Some results in this direction 
for the special case where $A$ is an interval have been obtained 
by Fujiwara \cite{fujiwara} and Schmidt \cite{schmidt}. The proof 
of theorem \ref{fABnogroup} below and conjecture at the end of the paper 
give additional links.

The goal of this paper is to communicate new sum-product estimates for polynomial
and rational functions. Our results are of two kinds: The first kind are valid 
even for small sets (of size $|A|>p^{\epsilon}$ for every $\epsilon>0$). The second 
kind extend Vu's characterization to a more general setting, but are valid 
only for large sets ($|A|> p^c$ for a fixed constant $0<c<1$). We thus 
expect the estimates for the small sets to be more useful in applications, 
whereas the large set results illuminate the general picture.

\paragraph{Small sets.}
The results in this section are stated only for sets of size $\abs{A}<\sqrt{p}$. 
Modification for large $\abs{A}$ involve no alterations in the fabric of the proofs,
but would introduce much clutter. Moreover,
for large $\abs{A}$, the large-set results are not only more general, but yield sharper 
quantitative estimates. We did not optimize the numeric constants that appear in the bounds 
below because the results are very unlikely to be sharp for any value of the constants. 

\begin{theorem}\label{sumthm}
Let $f\in \F_p[X]$ be a polynomial of degree~$d\geq 2$. Then for every set
$A\subset \F_p$ of size $\abs{A}\leq \sqrt{p}$ we have
\begin{equation*}
\abs{A+A}+\abs{f(A)+f(A)}\gg \abs{A}^{1+\frac{1}{16\cdot 6^d}}. 
\end{equation*}
\end{theorem}
Note that by Ruzsa's triangle inequality (Lemma~\ref{triangleineq} below)
this implies that for $\abs{A}\leq \sqrt{p}$ and
any polynomial $g$ of the form 
$g(x,y)=x+f(y)$ with $\deg f=d\geq 2$ we have $\abs{g(A)}\gg \abs{A}^{1+\frac{1}{32\cdot 6^d}}$,
which is a generalization of \cite[Theorem 3.1]{hart_li_shen_fourieragain}.

The next result is an extension of \eqref{vu_bound} to sets of any size for polynomials of degree two. 
\begin{theorem}\label{thmquadratic}
There exists an absolute constant $c>0$ such that whenever
$f\in \F_p[X,Y]$ is a bivariate quadratic polynomial that is
not of the form $f=g(ax+by)$ for some univariate polynomial $g$, then
for every $A\subset \F_p$ of size $\abs{A}\leq \sqrt{p}$
we have
\begin{equation*}
\abs{A+A}+\abs{f(A)}\gg \abs{A}^{1+c}.
\end{equation*}
\end{theorem}

Our final small-set result is another generalization of the sum-product theorem itself:
\begin{theorem}\label{prodthm}
Suppose $f=\sum_{i=1}^k a_i x^{d_i} \in\F_p[X]$ is a polynomial with $k$ terms,
and an integer $d\geq 2$ satisfies $d_i\leq d$ for all $i=1,\dotsc,k$.
Then for every positive integer $r$, and every set 
$A\subset \F_p$ of size $p^{4/r}d^{40r}\leq \abs{A}\leq \sqrt{p}$ we have
\begin{equation*}
\abs{AA}+\abs{f(A)+f(A)}\gg \abs{A}^{1+\veps},
\end{equation*}
where
\begin{equation*}
\veps=(5000(r+k)^2\log_2 d)^{-k}.
\end{equation*}
\end{theorem}
The main appeal of this estimate is that the dependence on the degree of $f$ is merely logarithmic, 
which suggests that the exponents in all the sum-product estimates should not depend on the 
degree. Further evidence that the 
exponents in sum-product results should not depend on the degree is provided by the sum-product estimates
for large subsets $A\subset \F_q$, which we present now.

\paragraph{Large sets.}
For a polynomial $f\in \F_q[X_1,\dotsc,X_n]$ and sets $A_1,\dotsc,A_n\subset \F_p$
write $N(f;A_1,\dotsc,A_n)$ for the number of solution to $f(x_1,\dotsc,x_n)=0$ in $x_i\in A_i$.
The commonly used case $N(f;A,\dotsc,A)$ will be abbreviated as $N(f;A)$. More generally,
if $V$ is a variety in $\A^n_{\F_q}$, then $N(V;A_1,\dotsc,A_n)$ is the number of points of $V$
on $A_1\times\dotsb\times A_n$.
The principal result that generalizes \eqref{vu_bound}
is 
\begin{theorem}\label{ApAcountingthm}
Let $f(x,y,z)$ be an irreducible polynomial of degree $d$ which is 
not of the form $P(ax+by,z)$ or $P(x,y)$. Moreover, let $A,B\subset\F_q$. Assume $d<q^{1/40}$. Then
\begin{equation*}
\abs{A+A} + \frac{\abs{B}\abs{A}^4}{N(f;A,A,B)^2}\gg
\begin{cases}
\abs{A}(\abs{A}/\sqrt{q})^{1/2}d^{-1},&
   \text{if }q^{1/2}\leq \abs{A}\leq d^{4/5}q^{7/10},\\
\abs{A}(q/\abs{A})^{1/3}d^{-1},& % The d^{-1} here is NOT a typo
   \text{if }\abs{A}\geq d^{4/5}q^{7/10}.
\end{cases}
\end{equation*}
In particular, the inequality \eqref{vu_bound} holds (with slightly worse dependence on $d$) 
as witnessed by setting $f(x,y,z)=g(x,y)-z$, $B=g(A)$ and
noting that $N(f;A,A,B)=\abs{A}^2$.
\end{theorem}
The condition that $f$ be irreducible is purely for convenience, since 
$\max(N(f_1),N(f_2))\leq N(f_1f_2)\leq N(f_1)+N(f_2)$ holds for reducible polynomials.
On the other hand, the condition that $f$ is not of the form $P(ax+by,z)$ is essential
because if $f$ is of this form, then for $A=\{1,\dotsc,n\}$ and an appropriate
$B$ the result fails. Though the theorem is formulated only for polynomials,
the questions about growth of rational function can be reduced to it, of which the
following result is an example.
\begin{theorem}\label{generalsum}
Let $f(x)\in \F_q(x)$, $g(x,y)\in \F_q(x,y)$ be non-constant rational functions of degree at most $d$, and 
assume $g(x,y)$ is not of the form $G(af(x)+bf(y)+c),G(x)$, or $G(y)$ 
with $a,b,c\in\F_q$. Then if $\abs{A}\geq q^\frac{1}{2}$ and $d<q^{1/50}$, 
we have the estimate
\begin{equation*}
|f(A)+f(A)| + |g(A,A)|\gg
\begin{cases}
\abs{A}(\abs{A}/\sqrt{q})^{1/2}d^{-2},&
   \text{if }q^{1/2}\leq \abs{A}\leq d^{8/5}q^{7/10},\\
\abs{A}(q/\abs{A})^{1/3}d^{-2},& % The d^{-2} here is NOT a typo
   \text{if }\abs{A}\geq d^{8/5}q^{7/10}.
\end{cases}
\end{equation*}
\end{theorem}

The main feature of all the results above is the abelian group structure 
inherent in $A+A$ and $AA$. That structure permits us to use sumset inequalities from the
additive combinatorics, as well as the Fourier transform. The following result shows that for most polynomials 
$f$, the set $f(A)$ always grows even in the absence of any group structure.
\begin{definition}
Let $f(x,y)\in \K[x,y]$ be a polynomial of degree $d$ in the $x$-variable. 
Call $f(x,y)$ \emph{monic} in the $x$ variable if the coefficient of the $x^d$ term is a non-zero constant of $\K$. That is,
\begin{equation*}
f(x,y) = cx^d + g(x,y),
\end{equation*}
where $c\in\K\setminus\{0\}$, and $g(x,y)$ is of degree $\leq d-1$ in the $x$-variable. 
\end{definition}
\begin{theorem}\label{fABnogroup} Let $f(x,y)\in \F_q[x,y]$ be a polynomial of degree $d$ which is non-composite, 
and is not of the form $g(x)+h(y)$ or $g(x)h(y)$. Suppose also that $f(x,y)$ is 
monic in each variable. Then if $\abs{A},\abs{B}\geq q^{7/8}$, 
\begin{equation*}
\abs{f(A,B)} \gg_d \min(\sqrt[3]{q\abs{A}\abs{B}}, \abs{A}^{3/4}\abs{B}^{3/4}q^{-7/16}).
\end{equation*}	
\end{theorem}
This result is to be compared with the estimate of Elekes and R\'onyai over
the real numbers: 
\begin{theorem}[\cite{elekes_ronyai}, Theorem~2] 
Let $f(x,y)\in\R(x,y)$ be a rational function of degree $d$ which is not 
of the form $G(g(x)+h(y))$, $G(g(x)h(y))$ or $G(\frac{g(x)+h(y)}{1-g(x)h(y)})$. Then there exists 
a constant $c=c(d)>0$ such that whenever $\abs{A}=\abs{B}=n$,
\begin{equation*}
\abs{f(A,B)}\gg_d n^{1+c}.
\end{equation*}
\end{theorem}
The statement appearing in \cite{elekes_ronyai} is quantitatively weaker, but the 
bound of $n^{1+c}$ follows from the proof. Note that the case $\frac{g(x)+h(y)}{1-g(x)h(y)}$
arises because $\R$ is not algebraically closed. Indeed, $\frac{x+y}{1-xy}=G(F(x)F(y))$
where $G(x)=\frac{x-1}{i\cdot (x+1)}$ and $F(x)=\frac{1+ix}{1-ix}$.

The rest of the paper is organized as follows. In sections \ref{sec_analtools} and \ref{sec_algtools}
we gather analytic and algebraic tools used in the paper. Theorems~\ref{sumthm}, \ref{thmquadratic} and
\ref{prodthm} on small-set estimates are proved in section~\ref{sec_smallsets}. 
All the large-set results, apart from Theorem~\ref{fABnogroup} on $f(A,B)$, are proved
in section~\ref{sec_largesets}, whereas Section~\ref{sec_fAB} is devoted to Theorem~\ref{fABnogroup}.
It is followed by the proofs of the algebraic lemmas used throughout the paper.
The paper ends with several remarks and a conjecture.

\section{Analytic tools}\label{sec_analtools}
We shall need a number of tools from additive combinatorics 
that we collect here.
\begin{lemma}[Ruzsa's triangle inequalities, \cite{ruzsa_survey}, Theorems~1.8.1 and 1.8.7]\label{triangleineq}
For every abelian group $G$ and every triple of sets $A,B,C\subset G$ we have
\begin{equation*}
\abs{A\pm C}\abs{B}\leq \abs{A\pm B}\abs{B\pm C},
\end{equation*}
where the result is valid for all eight possible choices of the signs.
\end{lemma}
Let $s*A=A+A+\dotsb+A$ where $A$ appears $s$ times as a summand. 
\begin{lemma}[Pl\"unnecke's inequality, \cite{ruzsa_survey}, Theorem~1.1.1]\label{plunnecke}
For every abelian group $G$ and every $A\subset G$ we have
\begin{equation*}
\abs{s*A-t*A}/\abs{A}\leq (\abs{A\pm A}/\abs{A})^{s+t},
\end{equation*} 
where the result is valid for either choice of the sign.
\end{lemma}
Let $\lambda\cdot A=\{\lambda a: a\in A\}$ be the $\lambda$-dilate of $A$. The
following result of the first author is used in the proof of Theorem~\ref{prodthm} to obtain the logarithmic
dependence on the degree.
\begin{lemma}[\cite{bukh_sumsdilates}, Theorem~3]\label{dilatesupper}
If $\Gamma$ is an abelian group and $A\subset \Gamma$ is a finite set satisfying
$\abs{A+A}\leq K\abs{A}$ or $\abs{A-A}\leq K\abs{A}$, then 
$\abs{\lambda_1\cdot A+\dotsb+\lambda_k\cdot A}\leq K^P\abs{A}$, where
\begin{equation*}
P=7+12\sum_{i=1}^k \log_2(1+\abs{\lambda_i}).
\end{equation*}
\end{lemma}
\begin{lemma}[Szemer\'edi-Trotter theorem for $\F_p$, \cite{bkt}]\label{sztrot}
Let $\mathcal{P}$ and $\mathcal{L}$ be families of points and lines
in $\F_p^2$ of cardinality $\abs{\mathcal{P}},\abs{\mathcal{L}}\leq N\leq p^{2-\alpha}$ with $\alpha>0$. 
Then we have
\begin{equation*}
\abs{\{(p,l)\in \mathcal{P}\times\mathcal{L} : p\in l\}}\ll N^{3/2-\veps}
\end{equation*} 
for some $\veps=\veps(\alpha)>0$ that depends only on~$\alpha$. 
\end{lemma}

For sets $A,B$ in an abelian group, and a bipartite graph $G\subset A\times B$
we put $A+_G B=\{a+b : (a,b)\in G\}$ to denote their sumset along~$G$.
Similarly, $A\cdot_G B=\{ab : (a,b)\in G\}$ will denote their productset along~$G$.
\begin{lemma}[Balog-Szemer\'edi-Gowers theorem, Lemma~4.1 in \cite{sudakov_szemeredi_vu}]\label{bsg}
Let $\Gamma$ be an abelian group, and $A,B\subset \Gamma$ be two $n$-element sets.  
Suppose $G\subset A\times B$ is a bipartite graph with $n^2/K$ edges, and 
$\abs{A+_G B}\leq Cn$. Then one can find subsets $A'\subset A$ and $B'\subset B$ such
that $\abs{A'}\geq n/(16K^2)$, $\abs{B'}\geq n/(4K)$, and $\abs{A'+B'}\leq 2^{12}C^3K^5n$.
\end{lemma}

In the proofs of Theorems~\ref{ApAcountingthm} and \ref{generalsum} we will use
Fourier transform on $\F_q^n$. For that we endow $\F_q^n$ with probability measure, and its
dual with the counting measure. Thus, the Fourier transform is defined by
$\hat{f}(\xi)=\tfrac{1}{q^n}\sum_x f(x)\exp(2\pi i x\cdot \xi)$, Plancherel's theorem asserts that
\begin{equation*}
\tfrac{1}{q^n}\sum_x f(x)\overline{g(x)} = \innq{n}{f,g} = \inndq{n}{\hat{f},\hat{g}}=
\sum_{\xi} \hat{f}(\xi)\overline{\hat{g}(\xi)},
\end{equation*}
the convolutions are defined by $(f*g)(y)=\frac{1}{q^n}\sum_x f(x)g(y-x)$, and satisfy
$\widehat{f*g}=\hat{f}\hat{g}$.

\section{Algebraic tools}\label{sec_algtools}
Here we record several results in algebraic geometry that are repeatedly used 
throughout the paper.

Throughout the paper $\A_{\K}^n$ denotes the $n$-dimensional 
affine space over $\overline{\K}$, the algebraic closure of $\K$.
When the field is clear from the context, we write simply $\A^n$. 
For a reducible variety $V\subset \mathbb{P}^N$, define the \emph{total degree} $\deg(V)$  to be the
sum of the degrees of all irreducible components of $V$.

\begin{lemma}[Generalized Bezout's theorem, \cite{bezout}, p.~223, Example~12.3.1]
\label{bezout}
Let $V_1$ and $V_2$ be two varieties in $\mathbb{P}^N$ and let $W$ be their intersection. Then
\begin{equation*}
\deg(W)\leq \deg(V_1)\deg(V_2).
\end{equation*}
\end{lemma}

Much of this paper is about proving non-trivial upper bounds on the number of points
on varieties in Cartesian products. Both for comparison, and because 
we need it several times in the proofs, we give an explicit ``trivial bound''. 
Recall that $N(V;A_1,\dotsc,A_n)$ stands for the
number of points of $V$ on $A_1\times\dotsb\times A_n$. The following result is a generalization
of the well-known Schwartz--Zippel lemma to varieties.

\begin{lemma}\label{schwartz_zippel_variety}
Let $V$ be an $m$-dimensional variety of degree $d$ in $\A^n$. Let
$A_1,\dotsc,A_n\subset \A^1$ be finite sets of the same size. Then
\begin{equation*}
N(V;A_1,\dotsc,A_n)\leq d\abs{A_1}^m.
\end{equation*}
\end{lemma}
\begin{proof}
The proof is by induction on $m$, the case $m=0$ being trivial.
If $V$ is reducible, then its degree is the sum of the degrees of its components. Thus, we can
assume that $V$ is irreducible. If $d=1$, then $V$ is a hyperplane, and the lemma is immediate.
If $d\geq 2$, then the irreducibility of $V$ implies that for each $a\in A_1$ the hyperplane
$H_a=\{x_1=a\}$ intersects $V$ in a variety of dimension $m-1$. Indeed, if $\dim V\cap H_a=m$, then
$V\cap H_a$ is a component of $H_a$ contradicting irreducibility. By Bezout's theorem (Lemma~\ref{bezout}) 
the degree  of $V\cap H_a$ is at most~$d$, 
which by induction implies that 
\begin{equation*}
N(V;A_1,\dotsc,A_n)=\sum_{a\in A_1} N(V\cap H_a;A_2,\dotsc,A_n)\leq \abs{A_1}\cdot d\abs{A_2}^{m-1}.\qedhere
\end{equation*}
\end{proof}

Let $\K$ be an algebraically closed field of characteristic $p>0$. 
A rational function $f(x,y)\in \K(x,y)$ is called \emph{composite} 
if there exist rational functions $Q(u),r(x,y)$ with $Q(r(x,y))=f(x,y)$ and $\deg Q\geq 2$.
\begin{lemma}[Bertini--Krull theorem, \cite{bertini_krull}, p.~221]\label{bertinikrull}
A polynomial $f(x,y)\in \K(x,y)$ of degree at most $p-1$ is composite if and only if 
for a generic (cofinite) set $t\in \K$, the variety $f(x,y)=t$ is reducible.
\end{lemma}

We also use the explicit bound on the Fourier transform over a curve due to Bombieri:
\begin{lemma}[Theorem~6, \cite{bombieri_expsum}]\label{katz_exponential_bound}
Let $P(x_1,x_2)$ be a polynomial of degree $d$ over $\F_q$ without linear factors, and define 
the set $S\subset\F_q^2$ to be the zero set of $P$. If $\xi$ is any non-zero additive character 
of $\F^2_q$, we have the following bound:
\begin{equation*}
\left|\sum_{\vec{x}\in S}\xi(\vec{x})\right|\leq 2d^2q^{\frac12}.
\end{equation*}
\end{lemma}

\section{Small sets}\label{sec_smallsets}
We first prove that either $A+A$ or $f(A)+f(A)$ grows. The proof is
inspired by Weyl's differencing method for estimating exponential
sums. For $\deg f=2$ differencing reduces the problem to the standard sum-product
estimates, whereas for $\deg f\geq 3$, differencing lets us replace 
$f$ by a polynomial of lower degree. 
\begin{proof}[Proof of Theorem~\ref{sumthm}]
The proof is by induction on $k$. We could use theorem~\ref{thmquadratic}
as our base case, but its proof is particularly simple when
restricted to this special case, so we give it here. 

Suppose $f(x)=ax^2+bx+c$ is a polynomial of degree two, which is to say $a\neq 0$. 
Assume that $\abs{A+A}+\abs{f(A)+f(A)}\leq \Delta\abs{A}$.
Then by the 
triangle inequality (Lemma~\ref{triangleineq}) we have
$\abs{f(A)-f(A)}\leq \abs{f(A)+f(A)}^2/\abs{A}\leq \Delta^2\abs{A}$ and
$\abs{A-A}\leq \abs{A+A}^2/\abs{A}\leq \Delta^2\abs{A}$. 
As $f(x)-f(y)=a(x-y)(x+y+b/a)$,
it follows that there are at least $\abs{A}^2$ solutions to
\begin{equation*}
uw=z,\qquad\text{with }u\in A-A,\ w\in A+A+b/a,\ z\in (1/a)\cdot(f(A)-f(A)).
\end{equation*}
By the Balog-Szemer\'edi-Gowers theorem (Lemma~\ref{bsg}) applied to
$A-A$ and $A+A+b/a$ in the multiplicative group $\F_p^*$,
we infer that there are $A_1\subset A-A$ and $A_2\subset A+A+b/a$
of sizes $\abs{A_1}\geq \Delta^2\abs{A}/(16\Delta^8)$ and $\abs{A_2}\geq \Delta^2\abs{A}/(4\Delta^4)$ 
such that $\abs{A_1A_2}\leq 2^{12}(\Delta^4)^5\Delta^2\abs{A}$, and hence
\begin{equation*}
\abs{A_2A_2}\leq 2^{24}\Delta^{44}\abs{A}\leq 2^{26}\Delta^{46}\abs{A_2}
\end{equation*}
by the triangle inequality.
Since
$A_2+A_2\subset A+A+A+A+2b/a$, Pl\"unnecke's inequality (Lemma~\ref{plunnecke}) 
implies 
\begin{equation*}
\abs{A_2+A_2}\leq 
\abs{A+A+A+A}\leq \Delta^4\abs{A}\leq 4\Delta^6\abs{A_2}.
\end{equation*}
These inequalities contradict \eqref{sumproductreview} unless
$2^{26}\Delta^{46}\gg \abs{A_2}^{1/12}\geq (\abs{A}/4\Delta^2)^{1/12}$. Simple
arithmetic gives $\Delta\gg \abs{A}^{1/554}$.

Hence, we know that whenever $\abs{A}\leq \sqrt{p}$, and 
$f$ is quadratic, we have $\abs{A+A}+\abs{f(A)+f(A)}\geq
C \abs{A}^{1+1/554}$ for some constant $0<C\leq 1$. We shall
prove that whenever $\deg f=d\geq 3$, we have $\abs{A+A}+\abs{f(A)+f(A)}\geq
C \abs{A}^{1+\frac{1}{16\cdot 6^d}}$ with the same constant $C$.
Suppose that $\deg f=d\geq 3$ and the claim has been proved for all
polynomials of degree~$d-1$. Assume that $\abs{A+A}+
\abs{f(A)+f(A)}\leq \Delta \abs{A}$. Let $t$ be any number 
having at least $\abs{A}^2/\abs{A-A}\geq \Delta^{-2}\abs{A}$
representations as $t=a_1-a_2$ with $a_1,a_2\in A$. 
Define $A'=\{a\in A : a+t\in A\}$ and $g(x)=f(x+t)-f(x)$.  From the choice
of $t$ it follows that $\abs{A'}\geq \Delta^{-2}\abs{A}$. Pl\"unnecke's inequality (Lemma~\ref{plunnecke})
tells us that
\begin{equation*}
\abs{g(A')+g(A')}\leq \abs{f(A)+f(A)-f(A)-f(A)}\leq
\abs{f(A)+f(A)}^4/\abs{A}^3\leq \Delta^4\abs{A}\leq \Delta^6 \abs{A'}.
\end{equation*}
We also have $\abs{A'+A'}\leq \abs{A+A}\leq \Delta\abs{A}\leq \Delta^3\abs{A'}$.
However, $g$ is a polynomial of degree $d-1$, and by the induction 
hypothesis this implies $\Delta^6\geq C \abs{A}^{\frac{1}{16\cdot 6^{d-1}}}$.
Since $C\leq 1$, we have $C^{1/6}\geq C$, and the induction step is complete.
\end{proof}
We note that the argument in \cite{bourgain_potpourri} for
$x^2+xy$ does not seem to generalize to an arbitrary 
quadratic polynomial, as that argument crucially depends on $x^2+xy$ being linear 
in $y$, and so our proof of Theorem~\ref{thmquadratic} is again based
on the idea of differencing.
\begin{proof}[Proof of Theorem~\ref{thmquadratic}]
Assume that $f(x,y)=ax^2+by^2+cxy+dx+ey$ is a non-degenerate quadratic polynomial, and 
$\abs{A+A}+\abs{f(A)}\leq \Delta\abs{A}$. 
The Cauchy--Schwarz inequality implies that the equation
\begin{equation*}
ax_1^2+by_1^2+cx_1y_1+dx_1+ey_1=ax_2^2+by_2^2+cx_2y_2+dx_2+ey_2,\qquad x_1,x_2,y_1,y_2\in A
\end{equation*}
has at least $\Delta^{-1}\abs{A}^3$ solutions. 
Changing the variables to $v_1=x_1+x_2$, $v_2=y_1+y_2$, $u_1=x_1-x_2$, $u_2=y_1-y_2$ we conclude 
that there are at least $\Delta^{-1}\abs{A}^3$ solutions to
\begin{equation*}
a u_1 v_1 + b u_2 v_2+\tfrac{1}{2}c(u_2 v_1+u_1 v_2)+du_1+eu_2=0,\qquad u_1,u_2\in A-A,\ v_1\in A+A.
\end{equation*}
Rewrite the equation as 
\begin{equation}\label{quadgn}
u_1(av_1+\tfrac{1}{2}cv_2+d)+u_2 (b v_2+\tfrac{1}{2}c v_1+e)=0,\qquad u_1,u_2\in A-A,\ v_1,v_2\in A+A.
\end{equation}
Consider $g(v_1,v_2)=\frac{av_1+\tfrac{1}{2}cv_2+d}{b v_2+\tfrac{1}{2}c v_1+e}$. Suppose $g$ is a constant
function. Then $ab-\tfrac{1}{4}c^2=0$, which implies that
there are $s,t\in \mathbb{F}_{p^2}$ such that $ax^2+by^2+cxy=(sx+ty)^2$. Hence $g(v_1,v_2)=
\frac{s^2 v_1+st v_2+d}{t^2 v_2+st v_1+e}$, and it is evident that $dx+ey$ is a constant multiple
of $sx+ty$, contradicting non-degeneracy of~$f$. We conclude that $g(v_1,v_2)$ cannot be a constant
function. 

Without loss of generality we assume that $g$ depends non-trivially on $v_1$.
Call $v\in A+A$ \emph{bad} if $g(v_1,v)$ is a constant function. As $v$ is bad
only when the linear functions in the numerator and denominator of $g$ are proportional,
there is at most one bad $v$, which we denote $v_{\operatorname{bad}}$. 
When specialized to $v_2=v_{\operatorname{bad}}$ equation \eqref{quadgn} 
takes the form $\alpha u_1+\beta u_2=0$, where $\alpha$ and $\beta$ are constants, that
are not simultaneously zero. Thus there are at most $\abs{A-A}\abs{A+A}$ solutions
to \eqref{quadgn} with $v_2=v_{\operatorname{bad}}$. 

Choose a value for $v_2$ for which there are at least $N=\Delta^{-1}\abs{A}^3/\abs{A+A}-\abs{A-A}$ solutions
to \eqref{quadgn} and which is not bad. Since we can assume that $\Delta<\abs{A}^{3/4}/2$, 
it follows that $N\geq \tfrac{1}{2}\Delta^{-1}\abs{A}^3/\abs{A+A}\geq \tfrac{1}{2}\Delta^{-2}\abs{A}^2$.
Let $L_{u,v_1}$ denote the line in $\F_p^2$ with the equation
\begin{equation*}
(u_1-u)(av_1+\tfrac{1}{2}cv_2+d)+u_2 (b v_2+\tfrac{1}{2}c v_1+e)=0.
\end{equation*}
Since $g$ is not a constant function of $v_1$, the slope of $L_{u,v_1}$ uniquely determines $v_1$, 
from which we conclude that the family 
\begin{equation*}
\mathcal{L}=\{ L_{u,v_1} : v_1\in A+A,\ u\in A-A\}
\end{equation*}
contains $\abs{A+A}\abs{A-A}$ distinct lines. Consider the point set $\mathcal{P}=\bigl((A-A)+(A-A)\bigr)\times (A-A)$.
Each solution to \eqref{quadgn} yields $\abs{A-A}$ incidences between $\mathcal{P}$
and $\mathcal{L}$, one for each value of~$u$. By Lemmas~\ref{triangleineq} and \ref{plunnecke} we have
$\abs{\mathcal{P}}\leq \Delta^{10}\abs{A}^2$ and 
$\abs{\mathcal{L}}\leq \Delta^3\abs{A}^2$. Szemer\'edi-Trotter theorem (Lemma~\ref{sztrot}) implies that
$\Delta^{-1}\abs{A-A}N\leq (\Delta^{10}\abs{A}^2)^{3/2-\epsilon}$. Since $\abs{A-A}\geq \abs{A}$, we are done.
\end{proof}

The sum-product estimate for $\abs{AA}+\abs{f(A)+f(A)}$ is more delicate than the results above. Similarly to 
the proof of Theorem~\ref{sumthm}, the aim 
is to use the upper bounds on $f(A)+f(A)+\dotsb+f(A)$ to obtain an upper bound on
$g(A)+g(A)$ for some simpler polynomial $g$. However, now `simpler' means `having fewer non-zero terms', and one needs
to add far more copies of $f(A)$ to obtain a simpler~$g$. If $A$ was a product set itself, $A=BC$, then
our aim would be to find $b_1,b_2,\dotsc\in B$ so that the polynomial $h(x)=f(b_1 x)+f(b_2 x)+\dotsb$ has fewer terms
than $f$ has. If $\abs{B}\geq p^{\veps}$ we can hope to use the pigeonhole principle to find $h(x)$ and $h'(x)$, in which
one of the terms is the same. Then the polynomial $h(x)-h'(x)$ would have fewer terms.
Unfortunately, it might happen such $h(x)$ and $h'(x)$ are always equal, in which case $h(x)-h'(x)=0$. 
However, as $BB$ is small, for each fixed $\lambda$
we can find a large set $B'\subset B$ and an element $g$
such that $gb^\lambda\in B$ for all $b\in B'$. Using $B'$ in place of $B$
then permits us to use not only the terms of the form $f(bx)$, but also of the form $f(gb^{\lambda} x)$. 
As it turns out, that 
suffices to complete the proof. Regrettably, $A$ is not necessarily a product set, 
and that requires us to work 
with multiplications along a graph, introducing additional technical complications.

The following lemma is used to find an analogue of $B'$ in the sketch above.
Recall that $\lambda\cdot A=\{\lambda a : a\in A\}$. 
\begin{lemma}\label{lemmadilatesintersect}
Suppose $\lambda_1,\dotsc,\lambda_r$ are non-zero integers, and $\Gamma$ is an abelian group.
Furthermore, assume that $A\subset \Gamma$ satisfies $\abs{A+A}\leq K\abs{A}$. Let
\begin{equation*}
P=43\sum_{i=1}^r \log_2(1+\abs{\lambda_i}).
\end{equation*}
Then there is a set $B\subset A$ of size $\abs{B}\geq \tfrac{1}{2}K^{-P}\abs{A}$ and elements $g_1,\dotsc,g_r$ 
such that for every $b\in B$ the set
\begin{equation}\label{fullintersect}
\{a\in A : a+\lambda_i\cdot b+g_i\in A\text{ for all }i=1,\dotsc,r\}
\end{equation}
has at least $\tfrac{1}{2}K^{-P}\abs{A}$ elements.
\end{lemma}
\begin{proof}
For given $g_1,\dotsc,g_r\in \Gamma$ define 
\begin{equation*}
S(g_1,\dotsc,g_r)=\sum_{b\in A} \#\{a\in A : a+\lambda_i\cdot b+g_i\in A\text{ for all }i=1,\dotsc,r\}.
\end{equation*}
Summing over all $(g_1,\dotsc,g_r)\in\Gamma^r$ we obtain
\begin{align*}
\sum_{g_1,\dotsc,g_r}S(g_1,\dotsc,g_r)&=\sum_{a,b \in A} 
\#\{(g_1,\dotsc,g_r)\in \Gamma^r : a+\lambda_i\cdot b+g_i\in A\text{ for all }i=1,\dotsc,r\}\\&=
\sum_{a,b \in A} \abs{A}^r=\abs{A}^{r+2}.
\end{align*}
Since $S(g_1,\dotsc,g_r)=0$ unless $g_i\in A-A-\lambda_i\cdot A$, there is a way
to choose $g_1,\dotsc,g_r$ so that 
\begin{align*}
S(g_1,\dotsc,g_r)&\geq \abs{A}^{r+2}\prod_{i=1}^r \abs{A-A-\lambda_i\cdot A}^{-1}\\
                 &\geq \abs{A}^2 K^{-\sum_{i=1}^r \bigl(7+12+12+12\log_2(1+\abs{\lambda_i})\bigr)}\\
                 &\geq K^{-P} \abs{A}^2
\end{align*}
by Lemma~\ref{dilatesupper} and the inequality $31+12\log_2(1+\lambda)\leq 43\log_2(1+\lambda)$ valid for $\lambda\geq 1$.
Having chosen $g_1,\dotsc,g_r$, define $B$ to be the set of all $b\in A$ for which
the set in \eqref{fullintersect} has at least $\tfrac{1}{2}K^{-P}\abs{A}$
elements. Since the elements $b\in A\setminus B$ contribute at most
$\tfrac{1}{2}K^{-P}\abs{A}^2$ to $S(g_1,\dotsc,g_r)$, the lemma follows.
\end{proof}
Let $p_t(x_1,\dotsc,x_n)=x_1^t+\dotsb+x_n^t$ be the $t$'th power sum polynomial.
\begin{lemma}\label{lemmapowersums}
Suppose $\K$ is a field, and $0<t_1<t_2<\dotsb<t_r<\operatorname{char} \K$ are integers.
Let $w\colon \A_{\K}^r\to \A_{\K}^r$ be given by $w(x_1,\dotsc,x_r)_i=p_{t_i}(x_1,\dotsc,x_r)$.
Furthermore, assume that there are sets $S_1,\dotsc,S_r\subset\K$ of size
$n$ each, and $S=S_1\times \dotsb\times S_r$ is their product.
Then there is a set $S'\subset S$ of size 
$\abs{S'}\geq \abs{S}-n^{r-1}\sum_i t_i$
such that for all $x\in S'$ the number of solutions to 
\begin{equation*}
w(x)=w(y)\qquad \text{ with }y\in S'
\end{equation*}
is at most $\prod_{i} t_i$. 
\end{lemma}
\begin{proof}
The Jacobian determinant of $w$ is
\begin{equation*}
J(x)=\det (\partial p_{t_i}/\partial x_j)_{ij}=\det (t_i x_j^{t_i-1})_{ij}.
\end{equation*}
The polynomial $J$ is of degree $\sum_i (t_i-1)$, and it is non-zero
since its degree in $x_i$ is $t_r-1<\operatorname{char} \K$. By Lemma~\ref{schwartz_zippel_variety}
$J$ vanishes in at most $n^{r-1}\sum_i (t_i-1)$ points of~$S$. Thus the set 
$S'=\{x\in S : J(x)\neq 0\}$ is of size $\abs{S'}\geq \abs{S}-n^{r-1}\sum_i t_i$. 

For each $x\in S'$ the variety
$V_x=w^{-1}(w(x))$ is zero-dimensional.
The bound on the number of points of $V_x$ then follows from Bezout's theorem (Lemma~\ref{bezout}).
\end{proof}
\begin{proof}[Proof of Theorem~\ref{prodthm}]
The proof is by induction on~$k$. Suppose $k=1$ and $f=ax^d$. Let $B=\{x^d : x\in A\}$.
Then $\abs{B}\geq \abs{A}/d$. The condition $\abs{AA}+\abs{f(A)+f(A)}\leq \Delta\abs{A}$ 
implies $\abs{BB}+\abs{B+B}\leq d\Delta\abs{B}$. Then \eqref{sumproductreview} implies
$d\Delta\geq (\abs{A}/d)^{1/12}$, which together with $\abs{A}\geq d^{20}$ establishes 
the base case.

Suppose $f=\sum_{i=1}^k a_i x^{d_i}$ and $k\geq 2$. If $g=\gcd(d_1,\dotsc,d_k)\neq 1$, then
upon replacing $A$ by $\{a^g : a\in A\}$ and $f$ by $\sum a_i x^{d_i/g}$ the problem reduces
to the case $\gcd(d_1,\dotsc,d_k)=1$. 
Assume $\abs{AA}+\abs{f(A)+f(A)}\leq \Delta\abs{A}$.
For $b\in \Gamma$ define $A_b=\{a\in A: ba\in A\}$.
Since $\gcd(d_1,\dotsc,d_k)=1$, there is a pair of exponents 
$d_i,d_j$ such that $d_i$ and $d_j$ are not a power of the same
integer. Without loss of generality we may assume that these
two exponents are $d_1$ and $d_2$. 
Let 
\begin{equation}\label{lambdadef}
\lambda_i=d_1^{r-i}d_2^i \text{ for }i=1,\dotsc,r.
\end{equation}
Then by the choice
of $d_1$ and $d_2$, all the $\lambda$'s are distinct.
Lemma~\ref{lemmadilatesintersect} then yields a set $B\subset A$ 
of size $\abs{B}\geq \tfrac{1}{2} \Delta^{-P}\abs{A}$
and elements $g_1,\dotsc,g_{r}$ such that 
if we define 
\begin{equation*}
A_b=\{a \in A: g_i b^{\lambda_i}a\in A\text{ for all }i=1,\dotsc,r\},
\end{equation*}
then for every $b\in B$
\begin{equation*}
\abs{A_b}\geq \tfrac{1}{2}\Delta^{-P}\abs{A},
\end{equation*}
where $P\leq 43r\log_2(1+d^{r})\leq 100 r^2\log_2 d$.
Thus,
\begin{align*}
\sum_{b_1,\dotsc,b_{r}\in A} \abs{A_{b_1}\cap\dotsb\cap A_{b_{r}}}
&=\sum_{a\in A} \bigl(\#\{b\in A: a\in A_b\}\bigr)^{r}\\
&\geq \frac{(\sum_{a\in A} \bigl(\#\{b\in A: a\in A_b\})^{r} }{\abs{A}^{r-1}}\\
&\geq (\tfrac{1}{2}\Delta^{-P})^{2r} \abs{A}^{r+1}.
\end{align*}

Define a graph $G$ as follows. Its vertex set is 
\begin{equation*}
V(G)=\{(b_1,\dotsc,b_r)\in B^r :
\abs{A_{b_1}\cap\dotsb\cap A_{b_{r}}}\geq 2^{-2r-1}\Delta^{-2Pr}\abs{A}\},
\end{equation*}
and the pair 
$(b_1,\dotsc,b_r)$, $(b_{r+1},\dotsc,b_{2r})$ is an edge of $G$ if
$\abs{A_{b_1}\cap\dotsb\cap A_{b_{2r}}}\geq 2^{-4r-3}\Delta^{-4Pr}\abs{A}$.
Since $\abs{V(G)}\abs{A}+\abs{B}^r (2^{-2r-1}\Delta^{-2Pr}\abs{A})\geq (\tfrac{1}{2}\Delta^{-P})^{2r} \abs{A}^{r+1}$,
the number of vertices in $G$ is $\abs{V(G)}\geq 2^{-2r-1}\Delta^{-2Pr} \abs{A}^{r}$.
Moreover, the graph $G$ contains no independent set of size $2^{2r+2}\Delta^{2Pr}$.
Indeed, if $\{(b_{i,1},\dotsc,b_{i,r})\}_{i=1}^m\subset V(G)$ is an independent set of size $m$, 
then 
\begin{align*}
\abs{A}\geq \left\lvert\bigcup_{i=1}^m A_{b_{i,1}}\cap\dotsb\cap A_{b_{i,r}}\right\rvert\geq 
m2^{-2r-1}\Delta^{-2Pr}\abs{A}-\binom{m}{2}2^{-4r-3}\Delta^{-4Pr}\abs{A},
\end{align*}
implying $m<2^{2r+2}\Delta^{2Pr}$.

For each $\mathbf{b}\in \F_p^r$ define $\mathbf{b}^\lambda=(b_1^\lambda,\dotsc,b_r^\lambda)$. 
Let $u(b)=(b^{d_1},\dotsc,b^{d_k})$, and for each
$\mathbf{b}=(b_1,\dotsc,b_r)\in V(G)$ define
$u(\mathbf{b})=u(b_1)+\dotsb+u(b_r)=(p_{d_1}(\mathbf{b}),p_{d_2}(\mathbf{b}),\dotsc,p_{d_k}(\mathbf{b}))$.
Now the argument breaks into two cases. In the first case, we will reduce
the problem to the case of a polynomial with at most $k-1$ terms, whereas 
in the second case, we will show that a certain sumset associated to $V(G)$
is too large.

\textbf{Case 1:} Suppose there is an index $i\in [r]$ and
$\mathbf{b}\mathbf{b}'\in E(G)$
such that $u(\mathbf{b}^{\lambda_i})\neq u(\mathbf{b}'^{\lambda_i})$, but
$u(\mathbf{b}^{\lambda_i})$ and $u(\mathbf{b}'^{\lambda_i})$ share a coordinate,
that is $p_{d_j}(\mathbf{b})=p_{d_j}(\mathbf{b}')$ for some~$j$. 
Then 
\begin{align*}
g(x)&=f(g_{i}b_1^{\lambda_i} x)+\dotsb+f(g_{i}b_r^{\lambda_i} x)-
     f(g_{i}b_1'^{\lambda_i}x)+\dotsb+f(g_{i}b_r'^{\lambda_i} x)
\\&=\sum_{j=1}^k a_jg_j^{d_j} (p_{d_j}(\mathbf{b}^{\lambda_i})-p_{d_j}(\mathbf{b}'^{\lambda_i}))x^{d_j}
\end{align*}
is a non-constant polynomial with one fewer term than~$f$. Moreover,
if we define $A'=A_{b_1}\cap\dotsb\cap A_{b_r}\cap A_{b_1'}\cap\dotsb\cap A_{b_r'}$
then $g(A')+g(A')\subset (2r)*f(A)-(2r)*f(A)$. By the definition of the graph $G$,
\begin{equation*}
\abs{A'}\geq \frac{\abs{A}}{2^{2r+3}\Delta^{2Pr}}.
\end{equation*}
A simple, but tedious calculation  
shows that if $\Delta$ is small enough for the conclusion
of the theorem to fail, then $\abs{A}\geq d^{40r} p^{4/r}$ implies 
$\abs{A'}\geq d^{40(r+1)} p^{4/(r+1)}$, permitting
us to apply the induction hypothesis with $r$ and $k$ 
replaced by $r+1$ and $k-1$ respectively. Then,
since
\begin{equation*}
\abs{g(A')+g(A')}\leq \Delta^{4r}\abs{A}=\Bigl(\Delta^{4r}\frac{\abs{A}}{\abs{A'}}\Bigr)\abs{A'},
\end{equation*}
the induction hypothesis implies that $\Delta^{4r}\frac{\abs{A}}{\abs{A'}}\geq \abs{A'}^\veps$,
where $\veps=(5000(r+k)^2\log_2 d)^{k-1}$. Another tedious calculation
shows that this implies the desired lower bound on $\Delta$.

\textbf{Case 2:}
Suppose that for all $i\in [r]$ and all edges
$\mathbf{b}\mathbf{b}'\in E(G)$ either all coordinates of 
$u(\mathbf{b}^{\lambda_i})$ and $u(\mathbf{b}'^{\lambda_i})$ 
are equal, or all of them are distinct. In particular,
$p_{d_1}(\mathbf{b}^{\lambda_i})=p_{d_1}(\mathbf{b}'^{\lambda_i})$ and 
$p_{d_2}(\mathbf{b}^{\lambda_i})=p_{d_2}(\mathbf{b}'^{\lambda_i})$ happen simultaneously.
Let 
\begin{equation*}
w(\mathbf{b})=(p_{d_1^r d_2}(\mathbf{b}),p_{d_1^{r-1}d_2^2}(\mathbf{b}),\dotsc,p_{d_2^{r+1}}(\mathbf{b})).
\end{equation*}
Since $p_{d_1}(\mathbf{b}^{\lambda_i})=p_{d_1^{r-i+1}d_2^i}(\mathbf{b})$ and 
$p_{d_2}(\mathbf{b}^{\lambda_i})=p_{d_1^{r-i}d_2^{i+1}}(\mathbf{b})$,
it follows that if $\mathbf{b}\mathbf{b}'\in E(G)$, then either $w(\mathbf{b})=w(\mathbf{b}')$ or
$w(\mathbf{b})$ differs from $w(\mathbf{b}')$ in all coordinates.

By Lemma~\ref{lemmapowersums} there is a set $X\subset V(G)$ of size
at least $\abs{V(G)}-\abs{A}^{r-1}rd^{r+1}\geq \tfrac{1}{2}\abs{V(G)}$ 
such that $\abs{X\cap w^{-1}(w(\mathbf{b}))}\leq d^{r^2+r}$ for all
$\mathbf{b}\in X$. Since $\abs{A}\geq d^{40r}$, and $2^r\geq r$ 
either $\Delta$ is sufficiently large to stop the argument here, or we have 
\begin{equation*}
\abs{X}\geq \abs{V(G)}-\abs{A}^{r-1}rd^{r+1}\geq \abs{A}^{r-1}(2^{-2r-1}\Delta^{-2Pr}\abs{A}-rd^{r+1})\geq 
\abs{A}^{r}2^{-2r-2}\Delta^{-2Pr}.
\end{equation*}
By the pigeonhole principle there is an $h\in \F_p$ such that
the hyperplane $H=\{x_1=h\}\subset \F_p^r$ contains $w(\mathbf{b})$ for
at least $\abs{X}/p$ values of $\mathbf{b}\in X$.
Pick $X'\subset X$ such that the points $\{w(\mathbf{b})\}_{\mathbf{b}\in X'}$ all lie in $H$ 
and are distinct, and such that  
\begin{equation*}
 \abs{X'}\geq \frac{1}{d^{r^2+r}}\cdot \frac{\abs{X}}{p}\geq  \frac{2^{-2r-2}\Delta^{-2Pr} \abs{A}^{r}}{p d^{r^2+r}}.
\end{equation*}
Since such an $X'$ is an independent set in $G$, we conclude
that $\abs{X'}\leq 2^{2r+2}\Delta^{2Pr}$, and the theorem follows.
\end{proof}

\section{Large Sets}\label{sec_largesets}
The results in this section require a more systematic use of the idea of differencing
appearing in the proof of Theorem~\ref{sumthm}, and used throughout the previous section.
The differencing is a special instance of a general strategy that consists in repeatedly 
applying the Cauchy--Schwarz inequality to increase the number of variables involved. Geometrically
a single application of the Cauchy--Schwarz corresponds to taking fiber products of two varieties. 
For illustration, consider the problem of showing that $f(A)$ grows. Let $g(x,y,z)=f(x,y)-z$. 
It suffices to show that $N(g;A,A,C)$ is much
smaller than $\abs{A}^2$ for every set $C\subset \F_q$ of size $\abs{C}=\abs{A}$. This is
a problem of bounding the number of points of a particular variety $V=\{g(x,y,z)=0\}$ on the 
Cartesian product $A\times B\times C$. The Cauchy--Schwarz inequality tells us that the number
of points of $V'=\{g(x,y_1,z_1)=0,g(x,y_2,z_2)=0\}$ on $A\times B\times B\times C\times C$ is at least 
$N(g;A,B,C)^2/\abs{A}$, and thus it suffices to establish a non-trivial upper bound for the number
of points on $V'$. The variety $V'$ is the fiber product of $V$ with itself for the projection map 
on the first coordinate. However, in general one can use less trivial fiber products.

The problem about sets, whether $f(A)$ is large, reduces to the problem about
counting points on the variety $\{f(x,y)-z=0\}$ on Cartesian products, but to pass in the 
other direction the only tool currently available is the Balog-Szemer\'edi-Gowers theorem (Lemma~\ref{bsg}),
which applies only for linear $f$. It is because of this we are forced to work with varieties
rather than sets. We note however that the analogues of the sumset inequalities
that were extensively used in the previous section are easy to prove with Cauchy--Schwarz
as above. For instance, let $g(x,y,z)=x+y-z$, and $A=B=C$, then 
the lower bound for the number of points on $V=\{x+y-z=0\}$ implies a lower bound
for the number of points on $V'=\{x+y_1-z_1=0,x+y_2-z_2=0\}$, which upon elimination of $x$ variable,
yields a lower bound for the number of points on $V''=\{y_1+z_2-y_2-z_1=0\}$. This relation
corresponds to the Pl\"unnecke inequality for $A+A-A$. For general analogues of 
Ruzsa's triangle inequalities and Pl\"unnecke's inequality see \cite{razborov_free_statistical}.
Another interesting example of a systematic use of fiber products in additive combinatorics is in \cite{katz_tao_kakeya}.

To establish the estimates on $N(f;A,A,B)$ promised in theorem \ref{ApAcountingthm} we will use
the special case of $N(f;A,A,A)$ as a stepping-stone. The general result will be deduced via an application of
Cauchy--Schwarz as explained above.

\begin{lemma}
Suppose $n\geq 2$ and $f(\vec{x})$ is an $n$-variable polynomial of degree $d$ with no linear factors, and 
$A\subset \F_q$. Then if $|A|\geq q^{1/2}$ and $d<q^{1/5n}$, 
we have the estimate
\begin{equation*}
\abs{A+A} + \frac{\abs{A}^n}{N(f;A)}\gg_n
\begin{cases}
\abs{A}(\abs{A}/\sqrt{q})^{1/2}d^{-1},&
   \text{if }q^{1/2}\leq \abs{A}\leq d^{4/5}q^{7/10},\\
\abs{A}(q/\abs{A})^{1/3}d^{-1/3},&
   \text{if }\abs{A}\geq d^{4/5}q^{7/10}.
\end{cases}
\end{equation*}
\end{lemma}
\begin{proof} We will first establish the case $n=2$, and then reduce the general case to it.

\textbf{Case $n=2$:} The proof is a standard Fourier-analytic argument. 
Namely, let $B=A+A$, and $S\subset \F_q^2$ be the set of solutions to 
$f(\vec{x})=0$, and let $C = S\cap A^2$ be the subset of the solutions with coordinates in $A$.  
Since $C+(A\times A)\subset B\times B$ it follows that
\begin{equation*}
\innq{2}{\chi_{B\times B} * \chi_{-A \times -A},\chi_S} \geq \frac{|C|\cdot|A|^2}{q^4}.
\end{equation*}
Using Plancherel's theorem, and evaluating the term $\xi=0$ separately, we obtain 
\begin{equation*}
\frac{|C|\cdot|A|^2}{q^4}\leq \inndq{2}{\hat{\chi}_{B\times B}\cdot\hat{\chi}_{-A \times -A},\hat{\chi}_S}=
\frac{\abs{B}^2\abs{A}^2}{q^5}+\sum_{\xi\neq 0}\hat{\chi}_{B\times B}(\xi)\hat{\chi}_{-A \times -A}(\xi)\hat{\chi}_S(-\xi).
\end{equation*}
Lemma~\ref{schwartz_zippel_variety} gives the bound $\abs{S}\leq dq$.
Moreover, since $f$ has no linear factors, by Lemma~\ref{katz_exponential_bound} 
$\hat{\chi}_S(\xi)\ll d^2 q^{-3/2}$. Therefore, by Cauchy--Schwarz and Parseval
\begin{align*}
\frac{|C|\cdot|A|^2}{q^4}&\ll d\frac{\abs{B}^2\abs{A}^2}{q^5} + d^2q^{-3/2}\inndq{2}{\abs{\hat{\chi}_{B\times B}}, \abs{\hat{\chi}_{-A \times -A}}}\\
                         &\leq d\frac{\abs{B}^2\abs{A}^2}{q^5} + d^2q^{-3/2}\norm{\hat{\chi}_{B\times B}}^{1/2}\norm{\hat{\chi}_{-A\times -A}}^{1/2}\\
                         &=d\frac{\abs{B}^2\abs{A}^2}{q^5}+d^2\frac{\abs{B}\abs{A}}{q^{7/2}}.
\end{align*}
Since $\abs{C}=N(f;A,A)$, we deduce that
either $\frac{\abs{A}^2}{N(f;A)}\abs{B}^2\geq \frac{1}{2d}q\abs{A}^2$ or 
$\abs{B}\frac{\abs{A}^2}{N(f;A)}\geq \frac{\abs{A}^3}{d^2q^{1/2}}$. In both cases, at least
one of $\abs{B}=\abs{A+A}$ and $\frac{\abs{A}^2}{N(f;A)}$ is as large as claimed, and the result follows.

\textbf{Case $n>2$:} Having established the $n=2$ case, we proceed by induction on~$n$. Note that we can suppose $f$ is irreducible. 
Now by Lemma~\ref{app_specialization}, one can pick a variable, say $x_1$, such 
that for all but $d^{n}(n-1)$ `bad' elements $c\in\F_q$ of $x_1$, the polynomial $f_c=f(c,x_2,x_3,\dots,x_n)$ 
has no linear factors. Then
\begin{equation*}
N(f;A)\leq d^{n}(n-1)|A|^{n-2} + \sum_{\textrm{`good' } c} N(f_c;A).
\end{equation*}
Now, if the first term on the RHS is bigger, we are done. Else, we have:
\begin{equation*}
\abs{A+A} + \frac{\abs{A}^n}{N(f,A)} \geq \abs{A+A} +\frac{\abs{A}^n}{2\max_{\textrm{good c}}N(f_c,A)},
\end{equation*}
and we are done by induction.
\end{proof}

\begin{lemma}\label{lem_vararb}
Let $W$ be an $m$-dimensional 
irreducible variety of degree $d\geq 2$ in $\A_{\F_q}^n$, and  
$A\subset \F_q$. 
Then if $|A|\geq q^{1/2}$ and $d<q^{1/5n}$, we have the estimate
\begin{equation*}
\abs{A+A} + \frac{\abs{A}^{m+1}}{N(W;A)}\gg_n
\begin{cases}
\abs{A}(\abs{A}/\sqrt{q})^{1/2}d^{-1},&
   \text{if }q^{1/2}\leq \abs{A}\leq d^{4/5}q^{7/10},\\
\abs{A}(q/\abs{A})^{1/3}d^{-1},& % The d^{-1} here is NOT a typo
   \text{if }\abs{A}\geq d^{4/5}q^{7/10}.
\end{cases}
\end{equation*}
\end{lemma}
\begin{proof}\label{var_cor}
By Lemma~\ref{app_projlemma}, we can find a subset of $m+1$ 
coordinates $y_1,y_2,..,y_{m+1}$ on which to project $W$ 
such that the image lies in an $m$-dimensional 
variety $W'$ of degree $d'\geq 2$, which is thus defined 
by an equation $f(\vec{y})=0$ for a polynomial $f$ of 
degree~$d'$. By Corollary~\ref{projcontrol} and Lemma~\ref{bezout} there is a proper subvariety $V\subset W$
of degree at most $d^3n$ on the complement of which the projection from $W$ to $W'$ 
is at most $(d/d')$-to-$1$. Hence, 
$N(W;A)\leq N(V;A)+(d/d')N(W';A)$. As $\dim V\leq m-1$, Lemma~\ref{schwartz_zippel_variety}
yields $N(V;A)\leq d^3n\abs{A}^{m-2}$. Thus if $(d/d')N(W';A)\leq N(V;A)$, then we are done. Otherwise,
\begin{equation*}
\abs{A+A}+\frac{\abs{A}^{m+1}}{N(W;A)}\geq\abs{A+A}+\frac{\abs{A}^{m+1}}{2(d/d')N(W';A)}
\end{equation*}
and the results follows from the previous lemma.
\end{proof}
For the case where $A$ is an interval, similar results to Lemma~\ref{lem_vararb} were obtained
by Fujiwara \cite{fujiwara} and Schmidt \cite{schmidt}.

The preceding lemma implies Theorem~\ref{ApAcountingthm} concerning the lower
bound on $\abs{A+A}+\frac{\abs{B}\abs{A}^4}{N(f;A,A,B)}$.
\begin{proof}[Proof of Theorem~\ref{ApAcountingthm}]
Let $V\subset \A_{\F_q}^5$ be the $3$-dimensional variety
\begin{equation*}
f(x_1,x_2,x_5)=f(x_3,x_4,x_5)= 0.
\end{equation*}
Note that $V$ is the fiber product of $f(x,y,z)=0$ with itself along
the projection to the last coordinate. 
By the Cauchy--Schwarz inequality, $N(V;A,A,A,A,B)\geq \frac{N(f;A,A,B)^2}{\abs{B}}$. 
Let $U\subset \A_{\F_q}^4$ be the projection of $V$ onto $\{x_5=0\}$, and 
$W$ be the Zariski closure of $U$. 
Note that for any point $u\in U$, one of two things can happen: 
the preimage of $u$ in $V$ consists of at most $d$ points, or the 
preimage in $V$ is all of $\A_{\F_q}^1$. 
But the latter can only happen for at most $d^2$ points, by Corollary~\ref{preimagedegree}.
Thus we arrive at the following upper bound on $N(f;A,A,B)$
\begin{equation*}
\frac{N(f;A,A,B)^2}{\abs{B}}\leq N(V;A,A,A,A,B)\leq dN(U;A)+d^2q.
\end{equation*}
Note that if $dN(U;A)\leq d^2q$, then the lemma follows. Thus we may assume 
\begin{equation*}
\frac{N(f;A,A,B)^2}{\abs{B}}\leq 2N(U;A)\leq 2N(W;A).
\end{equation*}
An upper bound on $N(W;A)$ will now follow from the preceding lemma. Since the $\deg W\leq d^2$,
the condition $d<q^{1/5n}$ of the lemma holds, and to apply the lemma it 
suffices to check that $W$ does not contain a hyperplane.

We argue by contradiction. Suppose a hyperplane $L=\{a_1x_1+a_2x_2+a_3x_3 +a_4x_4+a_5=0\}$ is a component of~$W$. 
Since the surfaces 
\begin{equation*}
S_c=\{f(x_1,x_2,c)=f(x_3,x_4,c)=0\}
\end{equation*}
fiber $W$, the varieties $S_c\cap L$ fiber $L$. So by the dimension count, $S_c\cap L$ is generically $2$-dimensional. Thus for a generic $c\in \A^1$,
$S_c$ has a component in $L$. This can happen only if $a_1x_1+a_2x_2$ is constant on a component of $f(x_1,x_2,c)=0$, which means that $f(x_1,x_2,c)=0$ contains a line with slope $s = \frac{-a_1}{a_2}$. But having slope $s$ 
is a Zariski-closed condition, and since $V$ is irreducible, 
all points in $V$ must have slope $s$, which contradicts the assumption that $f(x,y,z)$ is not of the form $P(a_1x+a_2y,z)$.

Therefore we can apply the previous lemma to get the bound 
\begin{equation*}
\abs{A+A} + \frac{\abs{B}\abs{A}^4}{N(f;A,A,B)^2}\gg
\begin{cases}
\abs{A}(\abs{A}/\sqrt{q})^{1/2}d^{-1},&
   \text{if }q^{1/2}\leq \abs{A}\leq d^{4/5}q^{7/10},\\
\abs{A}(q/\abs{A})^{1/3}d^{-1},& % The d^{-1} here is NOT a typo
   \text{if }\abs{A}\geq d^{4/5}q^{7/10}.
\end{cases}
\end{equation*}
\end{proof}

We conclude by presenting the deduction of Theorem~\ref{generalsum} on growth of $\abs{f(A)+f(A)}+\abs{g(A,A)}$
from Theorem~\ref{ApAcountingthm}.
\begin{proof}[Proof of Theorem~\ref{generalsum}]
Since $f(x)$ has at most $d$ poles, we may assume that $A$ contains none of them.

The idea in the proof is to create a polynomial having many solutions on $f(A)$ and apply the previous lemma. 
This motivates us to define fields $\K = \F_q(x,y)$ and $\K_f = \F_q(f(x),f(y))$. It is easy to see 
that $\K_f\subset\K$  and $[\K:\K_f]=\deg(f)^2$.

Let $H(x,y,t)$ be the minimal polynomial for $g(x,y)$ over $\K_f$. 
Note that $H(x,y,g(x,y)) = 0$. Write $H(x,y,t)$ = $S(f(x),f(y),t)$. 
Then $S(x,y,t)$ is an irreducible polynomial
with at least $\abs{f(A)}^2$ roots on $f(A)\times f(A) \times g(A,A)$. 
If $S(x,y,t)$ is not of the form $P(ax+by,t)$, Theorem~\ref{ApAcountingthm} yields
\begin{equation*}
\abs{f(A)+f(A)}+\frac{\abs{g(A,A)}\abs{f(A)}^4}{(\abs{f(A)}^2)^2}\gg
\begin{cases}
\abs{A}(\abs{A}/\sqrt{q})^{1/2}D^{-1},&
   \text{if }q^{1/2}\leq \abs{A}\leq D^{4/5}q^{7/10},\\
\abs{A}(q/\abs{A})^{1/3}D^{-1},& 
   \text{if }\abs{A}\geq D^{4/5}q^{7/10},
\end{cases}
\end{equation*}
where $D=\deg S$.  Since an indeterminate $3$-variable polynomial $S'$ of degree $D'$ has 
$\binom{D'+3}{3}$ coefficients, and the condition that the
rational function $S'(f(x),f(y),g(x,y))$ vanishes 
is a system of $\binom{4dD'+2}{2}$ linear equations,
there is a non-zero polynomial $S'$ satisfying $S'(f(x),f(y),g(x,y))=0$,
of degree at most $D'$, provided
\begin{equation*}
\binom{D'+3}{3}\leq \binom{4dD'+2}{2}+1.
\end{equation*}
Since $S$ is irreducible, $S$ divides $S'$ implying $D\leq D'\leq 12d^2$.

If $S(x,y,t) = P(ax+by,t)$, then it would force $P(af(x)+bf(y),g(x,y)) = 0$, implying that 
$g(x,y)$ is algebraic over $\F_q(af(x)+bf(y))$. If that is so, $g(x,y)$ is of the form $G(af(x)+bf(y)+c)$, $G(x)$, or $G(y)$ 
by Lemma~\ref{app_algebaic_char}.
\end{proof}

\section{Growth of \texorpdfstring{$f(A,B)$}{f(A,B)} for very large sets \texorpdfstring{$A, B$}{A, B}}
\label{sec_fAB}
In this section we prove Theorem~\ref{fABnogroup} on growth of $f(A,B)$ without any assumptions on $A+A$. Excluding 
several algebraic lemmas, whose role is auxiliary to the main flow of the argument, the proof is not long. However,
the shape of the proof might be mysterious without further explanations.

In the previous section we saw a way to use the smallness of $A+A$ to reduce the task of bounding $N(f;A)$ to estimating
Fourier transform of the curve $\{f=c\}$. The latter was achieved by invoking the celebrated Weil's bound. In the absence
of any assumption on $A+A$, we need to dispose of the Fourier transform. The motivation for our approach comes from
Gowers $U^2$ norm, which is a substitute for the Fourier transform in additive combinatorics. If instead of smallness
of $A+A$ we assume smallness of another polynomial $g(A,A)$, then treating $g$ as `addition' and
its `inverse' (which might exist only implicitly) as `subtraction', we can create an analogue of the $U^2$ norm
of~$g$. Thus we require a polynomial $g$ for which $g(A)$ is small. It turns out  
that often the polynomial $f$ itself can fulfill the role of $g$. Moreover, assuming smallness 
of $f(A)$ we can create many varieties $V$ for which $N(V)$ is small using fiber products as in 
the previous section.

The substitute for the Fourier transform is not enough if there is no analogue for Weil's bound. That final ingredient
comes from the bound on the number of points on irreducible varieties. It should be noted that we will not use 
the full strength of Deligne's bound on the number of such points, but rather an earlier theorem of Lang and Weil 
\cite{lang_weil} that is a consequence of Weil's work on curves.

An alternative perspective on the argument that follows is that it is about turning sums over large 
subsets $A\subset \F_q$ into complete sums over $\F_q$, which are easier to study by algebraic means. 
A desirable estimate on such a sum
is $\left|\sum_{x\in A} S(x)\right|\leq \left|\sum_{x\in\F_q} S(x)\right|$, but in general it holds only if $S(x)$ is 
\emph{positive}. The terms $S(x)$ that appear in our sums are not always positive, 
but the Cauchy--Schwarz inequality reduces the problem to a sum of positive terms, which can be completed:
\begin{equation*}
\left|\sum_{x\in A} S(x)\right|\leq |A|^{\frac{1}{2}}\left|{\sum_{x\in A} S(x)^2}\right|^{\frac{1}{2}}
\leq |A|^{\frac{1}{2}}\left|{\sum_{x\in\F_q} S(x)^2}\right|^{\frac{1}{2}}.
\end{equation*}
The sum on the right is now a complete sum.

\begin{proof}[Proof of Theorem~\ref{fABnogroup}]
Set $C=f(A,B)$. Suppose $\abs{C}\leq q/2$, for else there is nothing to prove. 
We start with the inequality
\begin{equation*}
\abs{A}^2\abs{B}^2\leq\displaystyle\sum_{f(x_1,y_2)=z_1, f(x_2,y_1)=z_2} A(x_2)B(y_2)C(z_1)C(z_2)C(f(x_1,y_1)),
\end{equation*}
where we identify sets with their characteristic functions, that is $A(t)=1$ if $t\in A$, and $0$ otherwise, 
and likewise for $B$ and $C$. 

As in the usual application of Gowers $U^2$ norm, we will need to replace one of the sets
by a function of mean zero. Let $S(t)=1$ for $t\in C$, and $S(t)=-\frac{\abs{C}}{q-\abs{C}}$
if $t\not\in C$. The function $S(t)$ is clearly of mean $0$.
%Since by Lemma~\ref{schwartz_zippel_variety} 
Since $f$ is monic in $y$, for each $x$ and $z$ there are at most $d$ solutions
to $f(x,y)=z$. Thus 
there are at most $d\abs{A}\abs{C}$ solutions to $f(x,y)=z$ with
$x\in A$, $y\in\F_q$ and $z\in C$, and we have
\begin{equation*}
\abs{A}^2\abs{B}^2 - d^2\frac{\abs{A}\abs{B}\abs{C}^3}{q-\abs{C}} \leq\sum_{f(x_1,y_2)=z_1, f(x_2,y_1)=z_2} A(x_2)B(y_2)C(z_1)C(z_2)S(f(x_1,y_1))
\end{equation*}
If $\abs{C}^3\geq \abs{A}\abs{B}q/4d^2$, we are done. Otherwise, the first term dominates and we obtain
\begin{equation*}
\abs{A}^2\abs{B}^2  \leq 2\displaystyle\sum_{f(x_1,y_2)=z_1, f(x_2,y_1)=z_2} A(x_2)B(y_2)C(z_1)C(z_2)S(f(x_1,y_1))
\end{equation*}
We proceed to `clone' variables by applying Cauchy--Schwarz with respect to 
$y_2,x_2$
to obtain
\begin{equation*}
  \abs{A}^4\abs{B}^4  \leq 4\sum_{y_2,x_2}\left(A^2(x_2)B^2(y_2)\right)
  \sum_{y_2,x_2}\left(\sum_{f(x_1,y_2)=z_1, f(x_2,y_1)=z_2}C(z_1)C(z_2)S(f(x_1,y_1))\right)^2,
\end{equation*}
where the inner sum on the right side ranges over 
$z_1,z_2,x_1,y_1$.
Expanding, we get
\begin{equation*}
\abs{A}^3\abs{B}^3  \leq 4\displaystyle\sum_{\substack{f(x_1,y_2)=z_1, f(x_2,y_1)=z_2\\ f(x_3,y_2)=z_3,f(x_2,y_3)=z_4}}C(z_1)C(z_2)C(z_3)C(z_4)S(f(x_1,y_1))S(f(x_3,y_3)).
\end{equation*}
We come to our final application of Cauchy--Schwarz to this sum, this time with respect to
$z_1,z_2,z_3,z_4$:
\begin{equation*}
\frac{\abs{A}^6\abs{B}^6}{N^4} \leq 16\sum_{\substack{f(x_1,y_2)=z_1, f(x_2,y_1)=z_2\\ f(x_3,y_2)=z_3,f(x_2,y_3)=z_4\\f(x'_1,y'_2)=z_1,
f(x'_2,y'_1)=z_2\\ f(x'_3,y'_2)=z_3,f(x'_2,y'_3)=z_4}}S(f(x_1,y_1))S(f(x_3,y_3))S(f(x'_1,y'_1))S(f(x'_3,y'_3)).
\end{equation*}
Or more succinctly,
\begin{equation*}
\frac{\abs{A}^6\abs{B}^6}{N^4} \leq 16\sum_{\substack{f(x_1,y_2)=f(x'_1,y'_2),f(x_2,y_1)=f(x'_2,y'_1)\\f(x_3,y_2)=f(x'_3,y'_2),f(x_2,y_3)=f(x'_2,y'_3)}}
S(f(x_1,y_1))S(f(x_3,y_3))S(f(x'_1,y'_1))S(f(x'_3,y'_3)).
\end{equation*}
We next split the sum into many. Geometrically, the summation is over the variety 
\begin{equation*}
V=\left\{
\begin{aligned}
f(x_1,y_2)&=f(x'_1,y'_2),f(x_2,y_1)=f(x'_2,y'_1)\\f(x_3,y_2)&=f(x'_3,y'_2),f(x_2,y_3)=f(x'_2,y'_3)
\end{aligned}
\right\}.
\end{equation*}
Define $\phi\colon V\to \A^4$ by $\phi(\dotsc)=(f(x_1,y_1),f(x_3,y_3),f(x'_1,y'_1),f(x'_3,y'_3))$,
and let $V_{t_1,t_2,t_3,t_4}\subset V$ be the variety $\phi^{-1}(t_1,t_2,t_3,t_4)$. 
Then we can rewrite the preceding inequality as
\begin{equation*}
\frac{\abs{A}^6\abs{B}^6}{N^4} \leq 16\sum_{t_1,t_2,t_3,t_4\in\F_q} S(t_1)S(t_2)S(t_3)S(t_4) N(V_{t_1,t_2,t_3,t_4};\F_q).
\end{equation*}

Heuristically, one expects such a complicated variety as $V_{t_1,t_2,t_3,t_4}$ to be usually irreducible and of dimension $4$, 
since it is given by $8$ equations in $12$ variables. If that was indeed the case, then by \cite{lang_weil}
$N(V_{t_1,t_2,t_3,t_4};\F_q)=q^4+O_d(q^{7/2})$. Since $S$ is of mean zero, that would give the estimate of the theorem.

Formally, we appeal to Lemmas \ref{Visirr} and \ref{irreducibilitylemma} that tell us that $V$ is $8$-dimensional, irreducible and there is a Zariski-dense open set $U\subset \A^4$ such that whenever $(t_1,t_2,t_3,t_4)\in U$,
the variety $V_{t_1,t_2,t_3,t_4}$ is $4$-dimensional and irreducible. Let $Y=\A^4\setminus U$.
Since $\dim \phi^{-1}(Y)\leq 7$, the variety $\phi^{-1}(Y)$ contains at most $O_d(q^7)$ points in $\F_q^{12}$
and thus
\begin{align*}
\frac{\abs{A}^6\abs{B}^6}{N^4}&\ll \sum_{t_1,t_2,t_3,t_4\in \F_q} S(t_1)S(t_2)S(t_3)S(t_4)N(V_{t_1,t_2,t_3,t_4};\F_q)\\
&= \sum_{t_1,t_2,t_3,t_4\in U\cap \F_q} S(t_1)S(t_2)S(t_3)S(t_4)(q^4+O_d(q^{7/2}))+
  O_d(q^7)\\
&\ll_d q^{7/2}\sum_{t_1,t_2,t_3,t_4\in U\cap F_q} \abs{S(t_1)S(t_2)S(t_3)S(t_3)}+O_d(q^7)\ll_d N^4 q^{7/2},
\end{align*}
where the most important step is using that $S(t)$ is of mean zero to pass from the second to the third line.
Thus $\frac{\abs{A}^6\abs{B}^6}{N^4}\ll_d N^4 q^{7/2}$, it follows that $N\gg_d \abs{A}^{3/4}\abs{B}^{3/4}q^{-7/16}$.
\end{proof}

\section{Proofs of Lemmas~\ref{Visirr} and \ref{irreducibilitylemma}}
Before starting, we introduce a piece of terminology. 
\begin{definition} 
An $n$-dimensional variety $V$ over an algebraically 
closed field $\K$ is said to be \emph{mainly irreducible}
if it has a unique irreducible component of dimension $n$.
\end{definition}

The following lemma will be very useful in showing that many fiber products are irreducible:
\begin{lemma}
Let $V, W$ be $n$-dimensional irreducible varieties over $\K$ with dominant, 
finite maps $f\colon V\rightarrow \A_{\K}^n$ and $g\colon W\rightarrow \A_{\K}^n$ 
of degrees prime to $p=\operatorname{char} \K$, and let $V',W'\subset \A_{\K}^n$ be the subsets over 
which $f$, resp.\ $g$ are unramified. Then if $V'\cup W' = \A_{\K}^n$, the 
fiber product $V\times_{\A_{\K}^n}W$ is irreducible. 
\end{lemma}
\begin{proof}
Let $\F=\K(V)$ and $\mathbb{G}=\K(W)$ be the function fields of $V$ and $W$ 
considered as finite extensions of $\M=\K(\A_{\K}^n)$. Now $V\times_{\A_{\K}^n}W$ 
is irreducible if and only if $\F$ and $\mathbb{G}$ are linearly 
disjoint over $\M$. Since $V'\cup W' = \A^n_{\K}$, 
the fields $\K$ and $\mathbb{G}$
have coprime discriminants over $\M$. 
Let $\F', \mathbb{G}'$ be the Galois closures of $\F$ and $\mathbb{G}$ over $\M$ respectively. 
Then $\F'$ and $\mathbb{G}'$ still have coprime discriminants, and so their 
intersection $\F'\cap \mathbb{G}'$ is an unramified extension 
over $\M$ with degree prime to $p$, which must be $\M$ itself, 
since $\A_{\K}^n$ has no nontrivial unramified extensions of degree 
prime to~$p$. But since $\F',\mathbb{G}'$ are Galois, this implies that 
they are linearly disjoint, and hence the subfields $\F$ and $\mathbb{G}$ are also linearly disjoint, as desired.
\end{proof}

Note that by looking at the map between generic points the above proof carries through over the part of 
$\A_{\K}^n$ where $f$ and $g$ are finite. We thus have the following easy corollary:

\begin{corollary}\label{irreducibilityfiber}
Let $V, W$ be $n$-dimensional, mainly irreducible varieties over $\K$ 
with dominant maps $f\colon V\rightarrow \A_{\K}^n$ and 
$g\colon W\rightarrow \A_{\K}^n$ of degree prime to $p,$ and let 
$V',W'\subset \A_{\K}^n$ be the subsets over which $f$, resp.\ $g$ are unramified. 
Then if $V'\cup W' = \A_{\K}^n$, the fiber product $V\times_{\A_{\K}^n}W$ is 
mainly irreducible of dimension $n$ with the possible exception of components that do 
not map dominantly to $\A_{\K}^n$.
\end{corollary}

From now on we adopt the notation of proof of Theorem~\ref{fABnogroup}, so let $V$ be the variety defined by the equations 
\begin{equation*}
\left\{f(x_1,y_2)=f(x'_1,y'_2),f(x_2,y_1)=f(x'_2,y'_1),f(x_3,y_2)=f(x'_3,y'_2),f(x_2,y_3)=f(x'_2,y'_3)\right\}.
\end{equation*}

Define also $f_1(x,y)$ and $f_2(x,y)$ to be the derivatives of $f$ 
with respect to the first and second coordinates respectively. Let $\deg_1 f$ and $\deg_2 f$ be degrees of $f$ in
the $x$ and $y$ variable, respectively. 
\begin{lemma}\label{Visirr}
Under the assumptions of Theorem~\ref{fABnogroup}, $V$ is an $8$-dimensional mainly irreducible variety.
\end{lemma}
\begin{proof}
Denote by $W$ the variety
\begin{equation*}
\left\{f(x_1,y_2)=f(x'_1,y'_2), f(x_3,y_2)=f(x'_3,y'_2)\right\},
\end{equation*}
and by $W'$ the variety 
\begin{equation*}
\left\{f(x_2,y_1)=f(x'_2,y'_1), f(x_2,y_3)=f(x'_2,y'_3)\right\}.
\end{equation*}
Since $V=W\times W'$, it suffices to show that $W$ and $W'$ are mainly irreducible of dimension $4$. We focus on $W$, the case of $W'$ being symmetric:

Consider the map $\phi\colon \A^3\rightarrow \A^2$ given by $\phi(r,s,t) = (f(r,t), f(s,t))$. Notice that $W$ is the fiber product $\A^3\times_{\A^2} \A^3$ of $\A^3$ with itself with respect to the map $\phi$, so to show that the $W$ 
is mainly irreducible it suffices by Lemma \ref{irreducibilitymap}
to show that the fibers of $\phi$ are $1$-dimensional and generically irreducible.

The fibers of $\phi$ are $\phi^{-1}(a,b)= \left\{f(r,t)=a, f(s,t)=b\right\}$. If we fix $t$, then by our assumptions on $f$ we get a finite, non-zero number of 
solutions in $r$ and $s$, so the fibers are one-dimensional. To see that they are generically irreducible, denote by $z_a$ the curve $f(r,t)=a$ and let 
$\pi_a\colon z_a\rightarrow \A^1$ be the projection map onto the second coordinate, that is onto $t$. 

Notice that $q^{-1}(a,b)=z_a\times_{\A^1}z_b$ with respect to the maps $\pi_a,\pi_b$. Since $f$ is not a composite polynomial, $z_a$ is generically irreducible by the Bertini--Krull theorem (Lemma~\ref{bertinikrull}). We are now in a position to 
apply Corollary~\ref{irreducibilityfiber}
By our assumption on $f$, $\pi_a$ is a finite map for all $a$. 
By the Jacobian criterion, the ramification locus of 
$\pi_a$ on the base is the set of $t$ for which there exists an $r$ with $f_1(r,t)=0$, and $f(r,t)=a$.

By our assumptions on $f$, for any fixed $r$ there are only finitely 
many $t$ such that $f_1(r,t)=0$ and so finitely many $a$ such 
that $r$ is in the bad locus of $\pi_a$. This implies that for 
generic $a,b,$ the maps $\pi_a$ and $\pi_b$ have disjoint bad loci and so we can 
apply Corollary~\ref{irreducibilityfiber}. This proves that $W$ is 4-dimensional and mainly irreducible. 
\end{proof}

Recall that $V_{t_1,t_2,t_3,t_4}$ is the variety
\begin{equation*}
\left\{\begin{aligned}
f(x_1,y_2)=f(x'_1,y'_2)&,f(x_2,y_1)=f(x'_2,y'_1)\\
f(x_3,y_2)=f(x'_3,y'_2)&,f(x_2,y_3)=f(x'_2,y'_3)\\
f(x_1,y_1)=t_1&,f(x_3,y_3)=t_2\\f(x'_1,y'_1)=t_3&,f(x'_3,y'_3)=t_4
\end{aligned}\right\}.
\end{equation*}

\begin{lemma}\label{irreducibilitylemma}
The $4$-dimensional family of varieties $V_{t_1,t_2,t_3,t_4}$ is generically $4$-dimensional and mainly irreducible.
\end{lemma}

\begin{proof}

Define $W$ to be the variety $\left\{f(x_1,y_2)=f(x'_1,y'_2),f(x_2,y_1)=f(x'_2,y'_1)\right\}$ and $W_{t_1,t_3}$ to be the variety
\begin{equation*}
 \left\{f(x_1,y_2)=f(x'_1,y'_2),f(x_2,y_1)=f(x'_2,y'_1), f(x_1,y_1)=t_1,f(x'_1,y'_1)=t_3\right\}.
\end{equation*}

We shall use the fact that $W_{t_1,t_3}$ has a canonical map $\pi_{t_1,t_3}\colon W_{t_1,t_3}\rightarrow\A^4$ given 
by projecting onto the coordinates $x_2,x'_2,y_2,y'_2$, 
and $V_{t_1,t_2,t_3,t_4}\cong W_{t_1,t_3}\times_{\A^4}W_{t_2,t_4}$ with respect to these maps.

\begin{lemma}
The varieties $W_{t_1,t_3}$ are $4$-dimensional and generically mainly irreducible.
\end{lemma}
\begin{proof}
First, since $f$ is non-composite, the variety $\left\{f(x_1,y_1)=t_1,f(x'_1,y'_1)=t_3\right\}$ is generically irreducible in $x_1,y_1,x'_1,y'_1$ by the Bertini--Krull theorem. Now if we can show that for generic $x_1,y_2,x'_1,y'_2$, the varieties 
$\left\{f(x_1,y_2)=f(x'_1,y'_2)\right\}$ and $\left\{f(x_2,y_1)=f(x'_2,y'_1)\right\}$ 
are $1$-dimensional and generically irreducible, then we will be done by an application of 
Corollary~\ref{irreducibilityfiber} with respect to the projection map onto $x_1,y_1,x'_1,y'_1$.

We will handle the variety $A_{x_1,x'_1}:= \left\{f(x_1,y_2)=f(x'_1,y'_2)\right\}$, the other one being symmetric. For each $x_1,x'_1,y_2$ we have a non-empty finite set of solutions for $y'_2$, so $A_{x_1,x'_1}$ is clearly 1-dimensional. 

Define a $3$-dimensional variety $A$ by $A=\{f(x_1,y_2)=f(x'_1,y'_2)\}$. Since $f$ is not composite,
$A$ is mainly irreducible. Consider the map 
$\phi\colon A\to \A^2$ given by projection onto the $x_1,x'_1$ coordinates. The fibers are precisely the $A_{x_1,x'_1}$. So by 
Lemma~\ref{genericirrfib}, $A_{x_1,x'_1}$ is generically irreducible if and only if the variety $A_{\times{\A^2}}A$, defined by
\begin{equation*}
f(x_1,y_2)=f(x'_1,y'_2),\ f(x_1,b_2)=f(x'_1,b'_2)
\end{equation*}
is mainly irreducible. But this is our $W'$ from Lemma~\ref{Visirr}, 
which we have already shown to be mainly irreducible. This completes the proof. 
\end{proof}

\begin{lemma}
$\pi_{t_1,t_3}$ is generically dominant.
\end{lemma}
\begin{proof} 
This is equivalent to proving that the Jacobian of $\pi_{t_1,t_3}$ does not vanish identically on $W_{t_1,t_3}.$  
The Jacobian is readily computed to be
\begin{equation*}
J=f_1(x_1,y_1)f_2(x_2,y_1)f_1(x'_1,y'_2)f_2(x'_1,y'_1)-f_1(x_1,y_2)f_2(x_1,y_1)f_1(x'_1,y'_1)f_2(x'_2,y'_1).
\end{equation*}
We have to show that $J$ does not vanish on $W$. Define $g(s,t)=\frac{f_1(s,t)}{f_2(s,t)}.$ 
Notice that none of the $f_1$ or $f_2$ terms are identically $0$ on $W$. 
Assume $J$ vanishes on $W$ for the sake of contradiction. 
Consider $J$ as a polynomial in $y_2$, $y'_2$. Then the assumption that $J=0$ on $W$
implies that on $A_{x_1,x'_1}:=\left\{f(x_1,y_2)=f(x'_1,y'_2)\right\}$ 
the function $\frac{f_1(x_1,y_2)}{f_1(x'_1,y'_2)}$ is a constant $C(x_1,x'_1)$. 

We showed above $A_{x_1,x'_1}$ is irreducible, which means that the function field 
of $A_{x_1,x'_1}$ is generated by $y_2,y'_2$ with $y'_2$ being of degree $\deg_2(f)$ over 
$\K(y_2)$. Since
\begin{equation*}
f_1(x_1,y_2)-C(x_1,x'_1)f_1(x'_1,y'_2)=0
\end{equation*}
and the degree of $f_1(x_1,y_2)-C(x_1,x'_1)f_1(x'_1,y'_2)$ in $y_2$ is less than $\deg_2(f)$, it must be $0$.
That implies that $f(x_1,y_2) = P(y_2)+ Q(x_1)$, which contradicts our original assumptions on $f$. 
So $J$ does not vanish identically on $W$.
\end{proof}

We are now almost ready to apply Lemma~\ref{irreducibilityfiber}, 
if we can show that $\pi_{t_1,t_3}\colon W_{t_1,t_3}\to \A^4$ has no `bad fixed locus'. That is, there is no divisor 
$D\in \A^4$ such that for all $t_1,t_2\in\K$, the map $\pi_{t_1,t_3}$ is either ramified over $D$, 
or is not finite over any point $d\in D$. 

\begin{itemize}
\item[\textbf{Case 1:}]
Suppose there is some divisor $D\in \A^4$ such that for 
all points $(y_2,x_2,y'_2,x'_2)\in D$, $\pi_{t_1,t_3}$ is not 
finite over $(y_2,x_2,y'_2,x'_2)$. Consider the projective closure $\bar{W}$ 
inside the space $\J^4_{x_1,y_1,x'_1,y'_1}\times\A^4_{y_2,x_2,y'_2,x'_2}$. We 
can extend $\pi$ to a map 
\begin{equation*}
\bar{\pi}\colon\overline{W}\to \A^4
\end{equation*}
which is 
now proper. 
Since $\pi_{t_1,t_3}$ is not finite over $\vec{x}\in\A^4_{y_2,x_2,y'_2,x'_2}$, the preimage
$\bar{\pi}_{t_1,t_3}^{-1}(\vec{x})$ has a point `at infinity'. Since this is true for all
$t_1,t_2$, it follows that $\bar{\pi}^{-1}(\vec{x})$ has 
a $2$-dimensional component at infinity.

The variety $\bar{\pi}^{-1}(y_2,x_2,y'_2,x'_2)$ is cut out by the projectivized equations
\begin{equation*}
\left\{f(x_1,y_2,e)=f(x'_1,y'_2,e),f(x_2,y_1,e)=f(x'_2,y'_1,e)\right\},
\end{equation*}
where $e$ is the homogenizing variable. Now, the component at infinity is given by $e=0$. 
By our assumption on $f$, the defining equations of $\bar{\pi}^{-1}(y_2,x_2,y'_2,x'_2)\cap\{e=0\}$ 
become 
\begin{equation*}
x_1^{\deg_1(f)}=x'^{\deg_1(f)}_1, y_2^{\deg_2(f)}=y'^{\deg_2(f)}_2.
\end{equation*}
This is a one-dimensional projective variety, which is a contradiction.

\item[\textbf{Case 2:}] Suppose that there is a divisor $D\in \A^4$ such that 
for all $t_1,t_3$, the projection $\pi_{t_1,t_3}$ is ramified over $D$. That is equivalent to 
saying that $\pi^{-1}_{t_1,t_3}(D)$ has a multiple component 
in $W_{t_1,t_3}$. Since this is true for all $t_1,t_3$, it 
implies that $\pi^{-1}(D)$ has a multiple component on $W.$ Now, $W$ is a 
direct product of $Y=\left\{f(x_1,y_2)=f(x'_1,y'_2)\right\}$ and $Z=\left\{f(x_2,y_1)=f(x'_2,y'_1)\right\}$. 
There are projections maps $\pi_{Y}\colon Y\to\A^2_{y_2,y'_2}$ and 
$\pi_{Z}\colon Z\to\A^2_{x_2,x'_2}$ such that $\pi=\pi_Y\times\pi_Z$. This means that over each
point $\vec{x}=(x_2,x'_2,y_2,y'_2)$ in $D$, either  
$Y_{y_2,y'_2}:= \left\{f(x_1,y_2)=f(x'_1,y'_2)\right\},$ or $Z_{x_2,x'_2}:=\left\{f(x_2,y_1)=f(x'_2,y'_1)\right\}$ 
has a multiple component. We treat the case of $Y_{y_2,y'_2}$ having a multiple component, the case of $Z_{x_2,x'_2}$ being symmetric.

If $Y_{y_2,y'_2}$ has $C$ as a multiple component, then $f_1(x_1,y_2)=f_1(x'_1,y'_2)=0$ 
on $C$. But by our assumptions on $f$ this is only a finite number of points, which is a contradiction. 
\end{itemize}

Applying Corollary~\ref{irreducibilityfiber}, we see that for generic $(t_1,t_2,t_3,t_4),W_{t_1,t_3}\times_{\A^4}W_{t_2,t_4}\cong V_{t_1,t_2,t_2,t_4}$ is $4$-dimensional and mainly irreducible \emph{except} for possibly over a 
proper closed subset $Y\subset\A^4.$ The following lemma rules out the existence of $Y$ and so completes the 
proof of Lemma~\ref{irreducibilitylemma}.

\begin{lemma} 
For generic $t_1,t_2,t_3,t_4$ there is no proper Zariski-closed subset 
$Y\subset\A^4$ with $W_{t_1,t_3}\times_{\A^4}W_{t_2,t_4}$ being more than $3$-dimensional over $Y.$
\end{lemma}
\begin{proof}
Note that the lemma is equivalent to the statement that for 
any $Y\subset \A^4$ of dimension at most $3$, we have 
\begin{equation*}
\dim\pi_{t_1,t_3}^{-1}(Y)+\dim\pi_{t_2,t_4}^{-1}(Y)-\dim(Y)\leq 4.
\end{equation*}
To prove this, first observe that for any point $\vec{y}\in\A^4$ 
the dimension of $\dim\pi_{t_1,t_3}^{-1}(\vec{y})$ is at most $1$. Moreover, 
since $W_{t_1,t_3}$ is $4$-dimensional and mainly irreducible 
and $\pi_{t_1,t_3}$ is dominant, the locus in $\A^4$ where the 
dimension of the fibers jump is at most $2$-dimensional. Therefore, 
all we have to exclude is the existence of a $2$-dimensional closed 
subvariety $Y\subset\A^4$ such that for almost all points $\vec{y}\in Y$ both 
$\dim\pi_{t_1,t_3}^{-1}(\vec{y})$ and $\dim\pi_{t_2,t_4}^{-1}(\vec{y})$ 
are $1$-dimensional. 
Since we want the result for generic
$t_1,t_2,t_3,t_4$, it suffices 
to exclude the case where a single bad variety $Y$ 
exists for all $t_1,t_3$. 
If such a $Y$ existed, then 
for almost all $\vec{p}\in Y, \dim\pi^{-1}(\vec{p})\geq 3$. 
However, since by assumption the polynomial $f(x,y)$ is 
monic in $x$, $\pi^{-1}(\vec{p})$ is a product 
of two non-degenerate curves, and so is $2$-dimensional. 
\end{proof}

\end{proof}

\section{Algebraic tidbits}
Often it is insufficient to know that some algebraic property holds generically, but
one needs a bound on the degree of the exceptional set. The next lemma and its corollaries
take care of this situation.
\begin{lemma}\label{projexceptlemma}
Suppose $V\subset \A^n$ is a variety of degree $d$, and $\pi\colon \A^n\to\A^m$ is the projection map.
Let $U=\{\vec{x}\in \pi(V) : \pi^{-1}(\vec{x})\cap V=\A^{n-m}\}$. Suppose $\dim U=\dim V+m-n-r$. Then $U$ is
contained in a variety of dimension $\dim U$ and degree at most $d^{r+1}$.
\end{lemma}
\begin{proof}
Think of $\A^n$ as $\A^m\times \A^{n-m}$, where $\pi(\vec{x},\vec{y})=\vec{x}$. Then for every
$\vec{y}_0\in \A^{n-m}$ write $V_{\vec{y}_0}=V\cap \{\vec{y}=\vec{y}_0\}$. Note that 
for a generic $\vec{y}_0$ the variety $V_{\vec{y}_0}$ is proper and of degree $d$.
We define varieties $W_1,W_2,\dotsc$ inductively.
Let $W_1=V_{\vec{y}_1}$ for some generic $\vec{y}_1\in\A^{n-m}$. Suppose 
$W_i$ has been defined, then
either for a generic $\vec{y}_{i+1}\in\A^{n-m}$ the inequality $\dim (W_i\cap V_{\vec{y}_{i+1}})<\dim W_i$ 
holds or there is an irreducible component $W'$ of $W_i$ of dimension $\dim W'=\dim W_i$ 
contained in $V_{\vec{y}_{i+1}}$ for every choice
of $\vec{y}_{i+1}\in\A^{n-m}$. In the former case let $W_{i+1}=W_i\cap V_{\vec{y}_{i+1}}$ for a 
generic $\vec{y}_{i+1}$, and continue the sequence. 
In the latter case, the sequence stops with $W_i$. In that case since
$U=\bigcap_{\vec{y}\in \A^{n-m}} V_{\vec{y}}$, we have $W'\subset U\subset W_i$. Thus,
$\dim U=\dim W_i$, and by Bezout's theorem
$\deg W_i\leq \prod \deg V_{\vec{y}_i}=d^i$. 
Since $\vec{y}_1\in \A^{n-m}$ is generic, $\dim W_1+(n-m)\leq \dim V$.
Finally, from
$\dim W_{i+1}<\dim W_i$, $\dim W_1\leq \dim V+m-n$ and $U\subset W_i$, it follows 
that the sequence of $W$'s terminates after at most $r+1$ elements.
\end{proof}
\begin{corollary}\label{preimagedegree}
If an irreducible polynomial $f(x,y,z)$ of degree $d$ is not of the form $g(x,y)$, then there are at most $d^2$ pairs $(a,b)$
for which $f(a,b,z)$ is zero as a polynomial in $z$. 
\end{corollary}
\begin{proof}
Let $V=\{f(x,y,z)=0\}$ and $\pi$ be the projection on $(x,y)$. Then in notation of the preceding lemma, $U$ is the set of 
pairs $(a,b)\in \A^2$ for which $f(a,b,z)=0$. Write $f(x,y,z)=\sum_i f_i(x,y)z^i$. The set $U$ is infinite if and only if
all the $f_i$ share a common factor, which is contrary to the assumption on $f$. Thus $\dim U=0$, and the result follows
from Lemma~\ref{projexceptlemma}.
\end{proof}
\begin{corollary}\label{linearfactorexcept}
Let $f$ be a polynomial of degree $d$ in $n$ variables, and  
suppose the polynomial $f_c(x_1,\dotsc,x_{n-1})=f(x_1,\dotsc,x_{n-1},c)$ has no linear factors for a generic~$c$. Then there are 
at most $d^{n}(n-1)$ values $c$ for which $f_c$ does have a linear factor.
\end{corollary}
\begin{proof}
Without loss of generality $f$ depends non-trivially on each of $x_1,\dotsc,x_n$.
Let 
\begin{equation*}
C=\{c : f_c\text{ has a linear factor }\}.
\end{equation*}
If $a_1x_1+\dotsb+a_{n-1}x_{n-1}+b$ is a factor of $f_c(x_1,\dotsc,x_{n-1})$,
then at least one of $a_i$ is non-zero. Thus, without loss of generality there is $C'\subset C$ of size 
$\abs{C'}\geq \abs{C}/(n-1)$ for which $f_c$ has a linear factor with non-vanishing coefficient $a_{n-1}$. By rescaling, 
we may assume that for every $c\in C'$ the linear factor is of the form $a_1x_1+\dotsb+a_{n-2}x_{n-2}+b-x_{n-1}$. 
Define polynomial $g$ in $2n-2$ variables by 
\begin{equation*}
g(x_1,\dotsc,x_{n-2},a_1,\dotsc,a_{n-2},b,c)=f(x_1,\dotsc,x_{n-2},a_1x_1+\dotsb+a_{n-2}x_{n-2}+b,c).
\end{equation*}
Since $f$ depends non-trivially on $x_{n-1}$, the polynomial $g$ depends non-trivially on $b$, thus the variety 
$V=\{g=0\}\subset \A^{2n-2}$ is of dimension $2n-3$.
Let $U=\{(a_1,\dotsc,a_{n-2},b,c) : g(x_1,\dotsc,x_{n-2},a_1,\dotsc,a_{n-2},b,c)=0\}$.
Since $(a_1,\dotsc,a_{n-2},b,c)\in U$ if and only if $a_1x_1+\dotsb+a_{n-2}x_{n-2}+b-x_{n-1}$ is a factor
of $f_c$, and $f_c$ has at least $1$ linear factor, it follows that $\abs{C'}\leq \abs{U}$.
As $C$ is finite, $\dim U=0$, and Lemma~\ref{projexceptlemma} implies $\abs{C}\leq (n-1)\abs{U}\leq (n-1)d^{n}$.
\end{proof}
\begin{corollary}\label{projcontrol}
Suppose $V\subset \A^n$ is an irreducible variety of degree $d$, the map $\pi\colon \A^n\to \A^m$ is the projection, and
$\dim V=\dim \pi(V)$. Let $U=\{ x\in V : \dim (\pi^{-1}(\pi(x))\cap V)>0\}$. Then $U$ is contained in a subvariety of $V$
of codimension $1$ and of degree at most $d^3(n-m)$.
\end{corollary}
\begin{proof}
Factor $\pi$ as $\pi=\sigma_{n-m}\dotsb \sigma_1$, where each $\sigma_i$ is a projection collapsing a single coordinate.
Let $\pi_i=\sigma_i\dotsb\sigma_1$ and 
\begin{equation*}
U_i=\{ x\in V : \sigma_i^{-1}(\pi_i(x))\cap V=\A^1\}.
\end{equation*}
Since $\dim \pi(V)\leq \dim \pi_i(V)\leq \dim V$, it follows $\pi_i(V)=\dim V$.
If $\dim U_i=\dim V$, then $U_i=V$ by irreduciblity of $V$, which would contradict $\dim \pi_i(V)=\dim V$.
Thus by Lemma~\ref{projexceptlemma} the degree of $\pi_i(U_i)$ is at most $d^2$.
Since $U$ is contained both in $V$ and in the union 
of $\pi_i^{-1}(\pi_i(U_i))$'s, the corollary follows
from Bezout's theorem (Lemma~\ref{bezout}).
\end{proof}

\begin{lemma}\label{app_specialization}
Suppose $n\geq 3$ and let $f(x_1,x_2,..,x_n)$ be an irreducible polynomial of degree $d$ 
with no linear factors over an algebraically closed field $\K$. 
Then there is a coordinate $x_i$ such that if we fix the value
of $x_i$ to an element $c\in \K $, then for all but $d^{n}(n-1)$
values of $c$, the resulting polynomial $f(x_1,x_2,\dots,x_{i-1},c,x_{i+1}\dots,x_n)$
also has no linear factors.
\end{lemma}
\begin{proof}
We will in fact show a stronger result that one can take $x_i$ to be one of $x_1,x_2,x_3$.
Since the roles of the variables are not symmetric, it is convenient to rename them
$x,y,z,w_1,\dotsc,w_{n-3}$.

Assume the conclusion of the lemma is false. Then, by Corollary~\ref{linearfactorexcept}, 
for all elements $c\in \K$, $f(c,y,z,w_1,\dotsc,x_{n-3})$ has a linear factor as a polynomial 
in $y,z,w_1,\dotsc,w_{n-3}$. Likewise for the $y$ and $z$ coordinates.

Moreover, these linear factors must have coefficients that are algebraic over $\K(x)$, say
$\alpha(x)y+\beta(x)z+\gamma_1(x)w_1+\dotsb+\gamma_{n-3}(x)w_{n-3}+\delta(x)$. 
The function $\alpha(x)$ vanishes 
only if $\partial_y f$ vanishes as well. Thus if $\alpha(x)$ vanishes infinitely often, then
$\partial_y f$ vanishes on a subvariety of $\{f=0\}$ of dimension $n$, which
by irreducibility of $f$ implies that $f$ does not depend on $y$. If $f$ does not
depend on $y$, then the lemma is trivially true. Thus we may assume that for a generic
$x$ the linear factor is of the form $-y+\beta(x)z+\gamma_1(x)w_1+\dotsb+\gamma_{n-3}(x)w_{n-3}+\delta(x)$.
Since $f$ is irreducible,
\begin{equation*}
f(x,y,z,w_1,\dotsc,w_{n-3})=\tau(x)\prod_j \bigl(-y+\beta_j(x)z+\gamma_{1,j}(x)w_1+\dotsb+\gamma_{n-3,j}(x)w_{n-3}+\delta_j(x)\bigr)
\end{equation*}
where the product is over the conjugates of $(\beta(x),\gamma_1(x),\dotsc,\gamma_{n-3}(x),\delta(x))$.

By the same reasoning applied to $z$ instead of $x$,
\begin{equation*}
f(x,y,z,w_1,\dotsc,w_{n-3})=\tau'(z)\prod_j \bigl(-y+\beta_j'(z)x+\gamma_{1,j}'(z)w_1+\dotsb+\gamma_{n-3,j}'(z)w_{n-3}+\delta_j'(z)\bigr).
\end{equation*}
Thus comparing these two descriptions of the  roots of $f$ considered as a polynomial in $y$ over 
$\overline{\K(x,z)}(w_1,\dotsc,w_{n-3})$, we conclude that $\gamma_{i,j}=\gamma_{i,j}'$
are constant, and  
\begin{equation*}
\beta(x)z+\delta(x)=\beta'(z)x+\delta'(z).
\end{equation*}
Thus $\beta,\beta',\delta,\delta'$ are linear, and $f$ is of the form
$f(x,y,z,w_1,\dotsc,w_{n-3})=Axz+Bx+Cy+Dz+\sum_i E_i w_i$. Since for fixed $y=c$ the polynomial
$f(x,c,z,\dotsc)$ has a linear factor, it follows that $f$ is itself linear, a contradiction. 
\end{proof}

\begin{lemma}\label{app_projlemma}
Let $W$ be a non-linear irreducible variety 
of dimension $m$ in $\A^n$, with coordinates 
being $x_1,x_2,..,x_n$. Then there are $m+1$ 
coordinates $x_{i_1},x_{i_2},...,x_{i_{m+1}}$ such 
that the projection of $W$ onto their span is contained in an $m$-dimensional 
non-linear hypersurface.
\end{lemma}
\begin{proof}
We induct on $n$. If $n=1$ or $n=2$, there is nothing to prove. Suppose $n\geq 3$,
and consider the $n$ coordinate hyperplanes, $\A^{n-1}_i=\{x_i=0\}$, 
and let $W_i$ be the projection of $W$ onto $\A^{n-1}_i$. Suppose there an $i$ such that $\dim W_i=m-1$. 
Then $W=W_i\times \A^1_i$ because $W$ is irreducible. Thus $W_i$ is non-linear, 
and by induction there is a projection of $W_i$ onto the span of 
$\{x_j\}_{j\in S}$. Then projection of $W$ onto the span of 
$\{x_j\}_{j\in S\cup \{i\}}$ is contained in a non-linear hypersurface. So, we may assume all the 
$W_i$ are of dimension~$m$.

Introduce a vector space structure on $\A^n$ in such a way that $0\in W$.
Let $L=\lspan W$ be the vector space spanned by $W$, and write 
$L_i$ for the projection of $L$ onto $\A^{n-1}_i$. Since $W$ is non-linear, $\dim L\geq \dim W+1=m+1$. 
If $\dim L_i=m$, then $i$'th basis vector $e_i$ is in $L$. If $\dim L_i=m$ for all $i$, then 
$\A^n=\lspan\{e_1,\dotsc,e_n\}\subset L$, implying $n=m+1$, in which case there is nothing to prove.
Thus, we can assume there is an $i$ such that $\dim L_i=m+1$. But then $W_i$ is non-linear, and the results follows
from the induction hypothesis. 
\end{proof}

\begin{lemma}\label{app_algebaic_char}
Let $f(x)$ be a non-constant rational function in $\F_q(x)$ of degree at most $q-1$. 
Suppose also there are non-zero rational functions $P(s,t)$ and $g(s,t)$ and constants $a,b\in \F_q$ such 
that $P(g(x,y),af(x)+bf(y)) = 0$. Then there exists a rational 
function $G(x)$, such that $g(x,y)$ is one of
\begin{equation*}
G(x),G(y),\text{ or } G(af(x)+bf(y)).
\end{equation*}
\end{lemma}
\begin{proof}
Let $\K$ be the subfield of $\F_q(x,y)$ consisting of all elements algebraic over $\F_q(af(x)+bf(y))$.
Note that $g(x,y)\in \K$ by assumption. Since $f$ is non-constant, $\K$ has transcendence 
degree $1$ over $\F_q$. We claim 
that $\K$ is isomorphic as a field to $\F_q(t)$. 
To see this, first note that $\K$ is finitely generated, since its a subfield of a finitely generated 
field. So $\K$ is the function field of a smooth, non-singular curve $C$ over $\F_q$. Also, the
embedding $\K\subset \F_q(x,y)$ corresponds to a dominant rational map from $\A_{\F_q}^2$ to $C$. 
But if $C$ was not birational to $\A^1$, then this map would have to be constant on every line, since a curve cannot map non-trivially to a curve of higher genus. But this contradicts that the map is dominant. So $\K$ is indeed generated by a single element. 

If $a$ or $b$ are 0, then $\K$ is generated by one of $y$ or $x$, and we are done. 
Suppose then that neither $a$ nor $b$ is $0$. Then it remains to prove that 
$\K$ is generated by $af(x)+bf(y)$ over $\F_q$, or equivalently that $af(x)+bf(y)$ is a non-composite rational function. That is, there are no rational functions $Q(t)\in\F_q(t)$, and $r(x,y)\in \F_q(x,y)$ such that
$Q(r(x,y)) = af(x)+bf(y)$ and $\deg(Q)>1$. Suppose for the sake of contradiction this is the case.

Since $Q(t)$ is a rational function of degree at least 2, $Q$ must be ramified over 
at least one finite point, say over $c\in\overline{\F_q}$. This means 
that $Q(t)-c$ has a double root at some point $c'\in\F_q$, so 
that $Q(r(x,y))-c = 0$ has a multiple component of the form $(r(x,y)-c')^2$, and 
$af(x) + bf(y) - c = 0$ must also have a multiple component. But by the Jacobian criterion, $af(x)+bf(y) - c = 0$ is only singular at points $(x_0,y_0)$ such that $f'(x_0)=f'(y_0)= 0$. There are only finitely many of these points, so $af(x)+bf(y)-c=0$ cannot have multiple components. This contradiction finishes the proof.
\end{proof}

\begin{lemma} \label{irreducibilitymap}
Let $p\colon V\to W$ be a dominant, equidimensional map such that $W$ is irreducible of dimension $m$, and for 
a generic point $\vec{w}\in W$, $p^{-1}(\vec{w})$ is irreducible of dimension $n$. Then $V$ has a unique irreducible component of 
dimension $m+n$.
\end{lemma}
\begin{proof}
That the dimension of $V$ is $m+n$ follows from dimension theory, so assume $V$ 
has two disjoint components of dimension $m+n$, $V_1$ and $V_2$ with 
$V_1\cup V_2=V$. For an open set $U\subset W$ we know 
that whenever $\vec{u}\in U$, $p^{-1}(\vec{u})$ is irreducible. 
So $p^{-1}(\vec{u})$ lies in either $V_1$ or in $V_2$. This means that 
$p(V_1)\cap p(V_2)$ is of dimension less than $m$. However, since the fibers of $p$ are of dimension $n$, dimension theory says that each of $p(V_1),  p(V_2)$ are of dimension at least $m$, and hence exactly $m$. But then $p(V_1), p(V_2)$ are two distinct components of $W$, contradicting the irreducibility of $W$.
\end{proof}
\begin{lemma}\label{genericirrfib}
Let $f\colon X\to Y$ be an equidimensional map with $X,Y$ irreducible, $\dim(Y)=m$, $\dim(X)=n.$ Then 
$f^{-1}(\vec{y})$ is generically mainly irreducible of dimension $n-m$ iff $X\times_Y X$ 
is mainly irreducible of dimension $2n-m$.
\end{lemma}
\begin{proof}
Let $\mathbb{L}=\overline{\F_q}(X), \K=\overline{\F_q}(Y)$. Since $f^{-1}(\vec{y})$ having at least two maximal reducible components is a Zariski-closed condition on $Y$, looking over the generic point we see that the theorem is equivalent to the following statement about fields:
\begin{equation*}
\Lbb\otimes_{\K}\mathbb{L}\textrm{ is a domain} \Longleftrightarrow \Lbb\otimes_{\K}\overline{\K}\textrm{ is a field}.
\end{equation*}

Call the above statements (i) and (ii) respectively, and consider the following additional statement: 
(iii) $\K$ is algebraically closed in $\Lbb$. We will show that both conditions are equivalent to (iii). 

\textbf{(ii)$\Longrightarrow$ (i)}. To prove $\Lbb\otimes_{\K}\Lbb$ is a domain, it is enough to show that 
$\Lbb\otimes_{\K}L\otimes_{\K}\displaystyle\overline{\K}$ is a domain, 
and the latter is
\begin{equation*}
\Lbb\otimes_{\K}\Lbb\otimes_{\K}\displaystyle\overline{\K}\cong\left(\Lbb\otimes_{\K}\overline{\K}\right)\otimes_{\displaystyle\overline{\K}}\left(\Lbb\otimes_{\K}\overline{\K}\right)
\end{equation*}
which is a domain, since the product of geometrically irreducible varieties is irreducible.

\textbf{(i)$\Longrightarrow$ (iii)}. Suppose (iii) fails to hold so that $\Lbb$ contains a finite algebraic extension 
$\M$ of $\K$ such that $\M\neq \K$. 
Then $ \Lbb\otimes_{\K}\overline{\K}$ contains a copy of 
$\M\otimes_{\K}\M$ which is not a domain, so that (i) fails to hold as well.

\textbf{(iii)$\Longrightarrow$ (ii)}.  Suppose not, so that $\K$ is algebraically closed in $\Lbb$, but $\Lbb\otimes_{\K}\overline{\K}$ is not a field. Since 
$\overline{\K}$ is a union of finite extensions of $\K$, there must exist some finite algebraic extension $\M$ of $\K$ such that $\Lbb\otimes_{\K}\M$ is not a field either. We can present $\M$ as 
$\M\cong\K[x]/(P(x))$ for some irreducible polynomial $P(x),$ so that $\Lbb\otimes_{\K}\M\cong \mathbb{L}[x]/(P(x)).$ 

Since $\Lbb[x]/(P(x))$ is not a field, $P(x)$ must factor as $P(x)=Q(x)R(x)$, where $Q,R$ are polynomials 
with coefficients in $\Lbb$.  But the coefficients of $Q$ and $R$ can be expressed as polynomials 
in the roots of $P$, and are therefore algebraic over $\K$. This contradicts the fact 
that $\K$ is algebraically closed in $\Lbb$.
\end{proof}

\section{Problems and remarks}
\begin{itemize}
\item
We expect Theorem~\ref{fABnogroup} to hold without the condition that $f$ is monic.
In fact, the proof 
presented above holds provided the irreducibility 
of $V$ and $V_{t_1,t_2,t_3,t_4}$ can be established, and it is there that the condition that $f$ is monic 
is invoked. Whereas the monicity condition on $f$ can be relaxed, it cannot be removed completely 
in view of the following counterexample due to Vivek Shende:

\begin{example}
Let $p(x),$ and $g(x,y)$ be two non-linear polynomials, such that $g(x,y)$ depends non-trivially 
on $y$, and $p(x)$ is not a square. Then for $$f(x,y)=p(x)g(x,y)^2,$$ the variety $V$ in the proof 
of Theorem~\ref{fABnogroup} is reducible.
\end{example}

Nonetheless, it seems likely that sum-product phenomenon should persist in the absence of any `group-like structure'.
More precisely, we believe in the following conjecture:

\begin{definition}
 Let $(G,+)$ be a one-dimensional abelian algebraic group 
(i.e.\ $G$ is either $\mathbb{G}_m$, $\mathbb{G}_a$, or an elliptic curve), and define $G_0$ to be 
\begin{equation*}
G_0=\left\{(g_1,g_2,g_3)\in G^3\mid g_1+g_2+g_3=0\right\}.
\end{equation*}
An irreducible surface $V\subset\A^3$ is said to 
be \emph{group-like} with respect to $G$ if 
there are irreducible curves $C_1,C_2,C_3$,
and an irreducible surface $W\subset C_1\times C_2\times C_3$,
and rational maps $f_i\colon C_i\to \A^1$, $g_i\colon C_i\to G$ such that
the Zariski closure of $(f_1\times f_2\times f_3)(W)$ is $V$
and the Zariski closure of $(g_1\times g_2\times g_3)(W)$ is $G_0$.
\end{definition}

\begin{conjecture}
There exists an absolute constant $\delta>0$ such that whenever
$f\in \F_p[x,y,z]$ is an irreducible polynomial
that depends non-trivially on all three variables, and $A$ a subset of 
$\F_p$, then either
\begin{equation*}
N(f;A,A,A)\ll_d \max \left(\abs{A}^{2-\delta}, \abs{A}^2(\abs{A}/p)^{\delta}\right)
\end{equation*}
or the surface $\{f=0\}$ is group-like with respect to some 
one-dimensional abelian algebraic group.
\end{conjecture}
Intuitively, the conjecture states that the only examples of surface $\{f=0\}$
containing many points on the Cartesian products are of the form 
\begin{equation*}
f(x,y,z)=F\bigl(G_x(x)\oplus G_y(y)\oplus G_z(z)\bigr)
\end{equation*}
for some group operation $\oplus$ and algebraic functions $F,G_x,G_y,G_z$.

\item
The proof of $\deg f=2$ case of Theorem~\ref{sumthm} can be modified to show that
$\abs{A+A}+\abs{f(A)+B}\gg \abs{A}\abs{B}^{1/1000}$ for every quadratic polynomial $f$
whenever $\abs{A},\abs{B}\leq \sqrt{p}$. The modification requires a sum-product estimate
on $\abs{A+A}+\abs{A\cdot_G B}$ where $G\subset A\times B$ is a dense bipartite graph. 
Such an estimate can be established by a simple modification of the proof in \cite{garaev_assymetric}
of the case $G=A\times B$. However, as the resulting proof is long and lacks novelty, it is
omitted from this paper, but can be found at \url{http://www.borisbukh.org/sumproductpoly_quadratic.pdf}.
It remains an interesting problem to show that $\abs{A+A}+\abs{f(A)+B}\gg \abs{A}\abs{B}^{\veps}$ for
some $\veps=\veps(\deg f)>0$ for polynomials of any degree.
\end{itemize}

\section*{Acknowledgments}
We thank Vivek Shende for many useful discussions concerning the algebraic part of the paper.
The relevance of \cite{bourgain_potpourri} was brought to our attention by Pablo Candela-Pokorna,
to whom we are grateful. We also thank Emmanuel Kowalski for a discussion on exponential sums,
and Igor Shparlinski for a careful reading of a preliminary version of this paper.

\bibliographystyle{alpha}
\bibliography{sumproductpoly}

\end{document}